\documentclass[12pt]{article}

\usepackage{amsmath,amsthm,amsfonts,amssymb, color}

\usepackage{tikz-cd}  

\newtheorem{theorem}{Theorem}[section]

\newtheorem{lemma}[theorem]{Lemma}

\def\cB{\mathcal{B}}

\def\cF{\mathcal{F}}

\def\cH{\mathcal{H}}

\def\cS{\mathcal{S}}

\def\bC{\mathbb{C}}

\def\bE{\mathbb{E}}
\def\bN{\mathbb{N}}
\def\bP{\mathbb{P}}

\def\bR{\mathbb{R}}

\def\e{\varepsilon}

\begin{document}

\title{Continuity in law for solutions of SPDEs \\
with space-time homogeneous Gaussian noise}

\author{Raluca M. Balan\footnote{University of Ottawa, Department of Mathematics and Statistics, 150 Louis Pasteur Private, Ottawa, Ontario, K1G 0P8, Canada. E-mail address: rbalan@uottawa.ca.} \footnote{Research supported by a grant from Natural Sciences and Engineering Research Council of Canada.}\and
Xiao Liang\footnote{University of Ottawa, Department of Mathematics and Statistics, 150 Louis Pasteur Private, Ottawa, Ontario, K1G 0P8, Canada. E-mail address: tlian081@uottawa.ca}
}

\date{May 15, 2023}
\maketitle

\begin{abstract}
\noindent In this article, we study the continuity in law of the solutions of two linear multiplicative SPDEs (the parabolic Anderson model and the hyperbolic Anderson model) with respect to the spatial parameter of the noise. The solution is interpreted in the Skorohod sense, using Malliavin calculus. We consider two cases: (i) the regular noise, whose spatial covariance is given by the Riesz kernel of order $\alpha \in (0,d)$, in spatial dimension $d\geq 1$; (ii) the rough noise, which is fractional in space with Hurst index $H<1/2$, in spatial dimension $d=1$. We assume that the noise is colored in time. The similar problem for the white noise in time was considered in \cite{bezdek16,GJQ}.
\end{abstract}


\vspace{1mm}


\section{Introduction}

Stochastic partial differential equations (SPDEs) are mathematical models for random phenomena which evolve in space and time. The study of SPDEs is a modern and active research field, with  difficult and interesting results being discovered every day. Key developments in this area focus on equations driven by space-time homogeneous Gaussian noise, using the random field approach initiated in Walsh' lecture notes \cite{walsh86}. Walsh' theory was developed for equations driven by a Gaussian
space-time white noise. In higher dimensions, the white noise is too rough, so that the solution can exist only in the
sense of distributions. To avoid this problem, Dalang 
introduced in \cite{dalang99} a new type of Gaussian noise that was smoother in space, but still
white in time. Dalang's seminal article sparked much interest in the scientific community, providing a natural extension of the classical It\^o theory, tailored for the study of SPDEs.
Together with Walsh' lecture notes, this article had created a new school of thought which was embraced by many researchers.

At the same time, in the late 1990's, a new process began to be used extensively in stochastic analysis as a more attractive model for the temporal structure of the noise than the Brownian motion. This process is the {\em fractional Brownian motion} (fBm), and is not a semi-martingale. One of the appealing feature of fBm is that it can be embedded into an isonormal Gaussian process for which one can use Malliavin calculus.
Placing the study of SPDEs in the context of Malliavin calculus can be traced back to \cite{bally-pardoux98}, and opened the door for deeper investigations about the probabilistic behaviour of the solutions, such as: H\"older continuity, Feynman-Kac representations, intermittency, existence and smoothness of density, exact asymptotic behavior of the moments, ergodicity, and Gaussian fluctuations for the spatial average. We refer the reader to \cite{BJQ,CDST,CD08,dalang-sanz09,FK09,hu-le19,HHLNT,HHLNT1,hu-nualart09,HNS11,HNV}
for a sample of relevant references.

In the present article, we examine another interesting property of solutions of SPDEs, namely the continuity in law with respect to the noise parameter. In the case equations driven by the white noise in time, the continuity in law of the solutions to the heat and wave equations driven by fBm of index $H \in (0,1)$ has been studied in the recent article \cite{GJQ}. This was a continuation of article \cite{GJQ1}, in which the authors considered the same problem for quasi-linear equations.  The general problem of weak continuity with respect to $H$ of the local time of fBm, or of various integrals with respect to fBm, was examined in \cite{JV1,JV2,JV3,JV4}. The continuity in law of the solution of the heat equation with white noise in time and spatial covariance given by the Riesz kernel was studied in \cite{bezdek16}. 

The purpose of the present article is to extend the results of \cite{bezdek16,GJQ} to the colored noise in time, and therefore contribute to the advancement of the rapidly growing knowledge about SPDEs, leading to a deeper understanding of the behaviour of the solutions to these equations.

\section{Framework and Main Results}

In this article, we consider the stochastic heat equation with linear multiplicative noise:
\begin{align}
\label{pam}
	\begin{cases}
		\dfrac{\partial u}{\partial t} (t,x)
		= \frac{1}{2}\Delta u(t,x) + u(t,x) \dot{W}(t,x), \quad
		t>0, \ x \in \bR^d, \\
		u(0,x) = 1,
	\end{cases}
\end{align}
and the stochastic wave equation with linear multiplicative noise:
\begin{align}
\label{ham}
	\begin{cases}
		\dfrac{\partial^2 u}{\partial t^2} (t,x)
		= \Delta u(t,x) + u(t,x) \dot{W}(t,x), \quad
		t>0, \ x \in \bR^d, \\
		u(0,x) = 1, \ \dfrac{\partial u}{\partial t} (0,x) = 0.
	\end{cases}
\end{align}
Equations \eqref{pam} and \eqref{ham} are called in the literature the {\em parabolic Anderson model}, respectively the {\em hyperbolic Anderson model}.

We assume that $W$ is a space-time homogeneous Gaussian noise. More precisely, $W=\{W(\varphi);\varphi \in C_0^{\infty}(\bR_{+} \times \bR^d)\}$ is a zero-mean Gaussian process, defined on a complete probability space $(\Omega,\cF,\bP)$, with covariance:
\begin{equation}
\label{cov}
\bE[W(\varphi)W(\psi)]=\int_{\bR_{+}^2 \times \bR^d} \cF \varphi(t,\cdot)(\xi) \overline{\cF \psi(s,\cdot)(\xi)}\gamma_0(t-s)dtds \mu(d\xi)=:\langle \varphi,\psi \rangle_{\cH},
\end{equation}
where $\cF \varphi(t,\cdot)(\xi)=\int_{\bR^d}e^{-i \xi \cdot x}\varphi(t,x)dx$ is the Fourier transform of the function $\varphi(t,\cdot)$, $\xi \cdot x$ is the Euclidean product in $\bR^d$, and $\gamma_0$ is a non-negative and non-negative definite function on $\bR_{+}$.

The spatial covariance structure of the noise is specified by the spatial spectral measure $\mu$. In general, this is assumed to be tempered, non-negative and non-negative definite. The scope of the present article is to analyze the influence of $\mu$ on the solutions of equations \eqref{pam} and \eqref{ham}, using continuity in law. To this end, we will consider two cases:
\[
\mu(d\xi)=
\left\{
\begin{array}{ll}
|\xi|^{-\alpha}d\xi  & \mbox{with $\alpha \in (0,d)$ and $d\geq 1$ ({\em Case I: the regular case})} \\
c_H|\xi|^{1-2H}d\xi & \mbox{with $H \in (0,1/2)$ and $d=1$ ({\em Case II: the rough case})}
\end{array} \right.
\]

In Case II, we assume that the constant $c_H$ is given by:
$$c_H=\frac{\Gamma(2H+1)\sin(\pi H)}{2\pi}.$$
With this choice, $W$ behaves in space like a {\em fractional Brownian motion} (fBm) of index $H$, in a sense which will be specified below. Recall that the fBm of index $H\in (0,1)$ is a zero-mean Gaussian process $B=(B_x)_{x \in \bR}$ with covariance:
\[
R_H(x,y):=\bE[B_x B_y]=\frac{1}{2}(|x|^{2H}+|y|^{2H}-|x-y|^{2H})=c_H \int_{\bR}\cF 1_{[0,x]}(\xi) \overline{\cF 1_{[0,y]}(\xi)}|\xi|^{1-2H}d\xi.
\]
If $H>1/2$, we have the following representation: for any $x>0$, $y>0$,
\[
R_H(x,y)=\alpha_H \int_0^x \int_0^y |u-v|^{2H-2}dudv,
\]
where $\alpha_H=H(2H-1)$. If $H=1/2$, $B$ is a Brownian motion.
The fBm has a modification whose sample paths are H\"older continuous of order less than $H$. These paths are rougher (i.e. less regular), or smoother (i.e. more regular) than the Brownian paths, depending on whether $H<1/2$ or $H>1/2$. Case II with parametrization $\alpha=2H-1$ and $d=1$ corresponds, modulo a constant, to the case when the noise $W$ behaves in space like a fBm of index $H >1/2$. This explains our terminology for the two cases.

\medskip

We mention now few more details about the noise. We denote by $\cH$ the Hilbert space defined as the completion of $C_0^{\infty}(\bR_{+} \times \bR^d)$ with respect to the inner product $\langle \cdot,\cdot \rangle_{\cH}$. Then $W$ can be extended to an isonormal Gaussian process $\{W(\varphi);\varphi \in \cH\}$, as defined in Malliavin calculus. We refer the reader to \cite{nualart06} for more details about Malliavin calculus.

The space $\cH$ may contain distributions, but it also contains some nice functions. Under certain conditions, it can be proved that $1_{[0,t] \times [0,x]} \in \cH$, in which case we can define the random field $\{W(t,x):=W(1_{[0,t] \times [0,x]}); t>0,x \in \bR^d\}$. In Case II, for any $t>0$ fixed, $\{W(t,x)\}_{x\in \bR}$ is, modulo a constant, a fBm of index $H$, since
\[
\bE[W(t,x) W(t,y)]=C_t R_H(x,y), \quad  \mbox{with $C_t=\int_0^t \int_0^t \gamma_0(u-v)dudv$}.
\]
The same thing is true in Case I with parametrization $\alpha=1-2H$ and $d=1$.

In Case I, the noise and the solution depend on the parameter $\alpha$, and will be denoted by $W^{\alpha}, u^{\alpha}$, respectively. The goal of this article is to prove the continuity in law of $u^{\alpha}$ with respect to $\alpha$, in the space of continuous functions on $\bR_{+} \times \bR^d$.
Similarly, in Case II, the noise and the solution are denoted by $W^{H},u^{H}$, respectively, and we are interested in the continuity in law of $u^{H}$ with respect to $H$. In both cases, we will provide constructions which will guarantee that all noise processes are defined on the same probability space, for all parameter values $\alpha$ (or $H$).

By the Bochner-Schwartz theorem, there exists a tempered measure $\mu_0$ on $\bR$ such that $\gamma_0$ is the Fourier transform of $\mu_0$ in the space $\cS_{\bC}'(\bR)$ of tempered distributions on $\bR$. Then,
\begin{equation}
\label{Fourier-phi}
\int_{\bR^2}\phi(t)\phi(s)\gamma_0(t-s)dtds =\int_{\bR}|\cF \phi(\tau)|^2 \mu_0(d\tau),
\end{equation}
for any $\phi \in \cS_{\bC}(\bR)$, where $\cS_{\bC}(\bR)$ is the space of $\bC$-valued rapidly decreasing $C^{\infty}$-functions on $\bR$, and $\cF \phi$ is the Fourier transform of $\phi$, given by $\cF \phi(\tau)=\int_{\bR}e^{-i \tau t}\phi(t)dt$ for all $\tau \in \bR$.

To construct all noise processes $(W^{\alpha})_{\alpha \in (0,d)}$, respectively $(W^{H})_{H \in (0,1/2)}$, on the same probability space, we will assume that
there exists a function $g_0:\bR \to [0,\infty]$ such that
\begin{equation}
\label{cond-mu0}
\mu_0(d\tau)=g_0(\tau) d\tau.
\end{equation}
A basic example is:
\begin{equation}
\label{def-gamma0}
\gamma_0(t)=\alpha_{H_0}|t|^{2H_0-2} \quad \mbox{for some $H_0 \in (\frac{1}{2},1)$},
\end{equation}
in which case, $g_0(\tau)=c_{H_0}|\tau|^{1-2H_0}$, with the same constants $\alpha_{H_0}$ and $c_{H_0}$ as above. In this case, $W$ behaves in time like a fBm of index $H_0$, in the sense that for any $x\in \bR^d$ fixed,
$\{W(t,x)\}_{t\geq 0}$ is, modulo a constant, a fBm of index $H_0$, since
\[
\bE[W(t,x)W(s,x)]=C_{x}'R_{H_0}(t,s) \quad \mbox{with $C_x'=\int_{\bR}|\cF 1_{[0,x]}(\xi)|^2 \mu(d\xi)$}.
\]

In what follows, we will need an extension of relation \eqref{Fourier-phi} to higher dimensions, namely:
\begin{equation}
\label{phi-high-dim}
\int_{\bR^n}\int_{\bR^n}\phi(\pmb{t_n})\phi(\pmb{s_n})
\prod_{j=1}^n \gamma_0(t_j-s_j)d\pmb{t_n}d\pmb{s_n} =\int_{\bR^n}|\cF \phi(\pmb{\tau_n})|^2 \prod_{j=1}^n g_0(\tau_j)d\pmb{\tau_n},
\end{equation}
for any $\phi \in \cS_{\bC}(\bR^n)$, where $\pmb{t}_n=(t_1,\ldots,t_n)$, $\pmb{s}_n=(s_1,\ldots,s_n)$, $\pmb{\tau_n}=(\tau_1,\ldots,\tau_n)$, and $\cF \phi$ is the Fourier transform of $\phi$, given by $\cF \phi(\pmb{\tau_n})=\int_{\bR}e^{-i \sum_{j=1}^n \tau_j t_j}\phi(\pmb{t_n})d\pmb{t_n}$ for all $\pmb{\tau_n} \in \bR^n$.

\medskip

We denote by $G_t$ the fundamental solution of the heat equation, respectively the wave equation, on $\bR_{+} \times \bR^d$. More precisely, in the case of the heat equation,
\[
G_t(x)=\frac{1}{(2\pi t)^{d/2}}\exp\left(-\frac{|x|^2}{2t} \right)
\]
In the case of the wave equation, in dimension $d=1$ or $d=2$, $G_t$ is an integrable function:
\begin{align*}
G_t(x)&= \frac{1}{2}1_{\{|x| <t\}} \quad \mbox{if} \ d=1\\
G_t(x)&= \frac{1}{2\pi}\frac{1}{\sqrt{t^2-|x|^2}}1_{\{|x|<t\}} \quad \mbox{if} \ d=2;
\end{align*}
in dimension $d=3$, $G_t$ is a measure given by $G_t=(4\pi t)^{-1}\sigma_t$ where $\sigma_t$ is the surface measure on the sphere $\{x \in \bR^3; |x|=t\}$; and in dimension $d\geq 4$, $G_t$ is a distribution.

The Fourier transform of $G_t$ plays an important role in this framework. We recall that in any dimension $d$, this is given by:
\[
\cF G_t(\xi)=
e^{-t|\xi|^2/2} \quad \mbox{in the case of the heat equation},
\]
respectively
\[
\cF G_t(\xi)= \frac{\sin(t|\xi|)}{|\xi|}
\quad \mbox{in the case of the wave equation}
\]

We will analyze simultaneously both equations, but we will provide a special analysis for the case of each equation, when needed.

\medskip

We say that $u$ is a {\bf Skorohod solution} of equation \eqref{pam}, respectively \eqref{ham}, if $u$ is adapted with respect to the filtration $\cF_t=\sigma\{W(1_{[0,t]}\varphi);\varphi \in C_0^{\infty}(\bR^d)\},t\geq 0$ associated with $W$, and satisfies the integral equation:
\[
u(t,x)=1+\int_0^t \int_{\bR^d}G_{t-s}(x-y)u(s,y) W(\delta s, \delta y),
\]
where the stochastic integral is interpreted in the Skorohod sense, i.e. it is given by the divergence operator $\delta$ with respect to $W$. We recall that $\delta$ is the adjoint of the Malliavin derivative $D$, and we refer the reader to \cite{nualart06} for the definitions and basic properties of these operators.

We begin now to review some basic facts about the solution. Using the methodology introduced in \cite{hu-nualart09}, we know that if it exists, the solutions of equations \eqref{pam} and \eqref{ham} have the series expansion:
\begin{equation}
\label{series}
u(t,x)=1+\sum_{n\geq 1}I_n\big(f_{t,x,n}\big) \quad \mbox{in} \quad L^2(\Omega),
\end{equation}
the terms of these series being orthogonal in $L^2(\Omega)$.
Here $I_n$ is the multiple integral of order $n$ with respect to $W$, and $f_{t,x,n}$ is given by:
\[
f_{t,x,n}(t_1,x_1,\ldots,t_n,x_n)=G_{t-t_n}(x-x_n)\ldots G_{t_2-t_1}(x_2-x_1)1_{\{0<t_1<\ldots<t_n<t\}}.
\]
The multiple integral $I_n(f)$ is a zero-mean random variable with finite variance, which is well-defined for all $f$ in the $n$-th tensor product space $\cH^{\otimes n}$. Its variance is given by $\bE|I_n(f)|^2=n!\|\widetilde{f}\|_{\cH^{\otimes n}}^{2}$, where
$\widetilde{f}$ is the symmetrization of the function $f$, defined by:
\[
\widetilde{f}(t_1,x_1,\ldots,t_n,x_n)=\frac{1}{n!}\sum_{\rho \in S_n}f(t_{\rho(1)},x_{\rho(1)},\ldots,t_{\rho(n)},x_{\rho(n)}),
\]
and $S_n$ is the set of permutations of $\{1,\ldots,n\}$.

The solution exists if and only if the series \eqref{series} converges in $L^2(\Omega)$, i.e.
\[
\sum_{n\geq 1}\bE|I_n\big(f_{t,x,n}\big) |^2 =\sum_{n\geq 1}n! \, \|\widetilde{f}_{t,x,n}\|_{\cH^{\otimes n}}^2<\infty.
\]

In Case I, for both heat and wave equations, a sufficient condition for the existence of the solution is {\em Dalang's condition}:
\[
\int_{\bR^d} \frac{1}{1+|\xi|^2}\mu(d\xi)<\infty,
\]
which is equivalent to
\begin{equation}
\label{dalang-cond}
d-\alpha<2.
\end{equation}
If in addition, $\gamma_0$ is given by the fractional kernel \eqref{def-gamma0}, then
\eqref{dalang-cond} is the necessary and sufficient condition for the existence of the Skorohod solution of the heat equation \eqref{pam} (see \cite{BC14}), while in the case of the wave equation \eqref{ham}, the necessary and sufficient condition for the existence of the Skorohod solution is: (see \cite{CDST})
\[
d-\alpha<2H_0+1.
\]

In the case of a rough noise in space (Case II) with temporal covariance $\gamma_0$ given by \eqref{def-gamma0}, a sufficient condition for the existence of the solution is $H_0+H>3/4$ for the heat equation (see \cite{hu-le19}), respectively $H>1/4$ for the wave equation (see \cite{song-song-xu20}).
These results provide extensions of the some earlier results of \cite{BJQ,HHLNT,HHLNT1} which were obtained for the white noise in time (corresponding to the case $H_0=1/2$).

The following two theorems are the main results of the present article.

\begin{theorem}[The Regular Case]
\label{main-th1}
Let $\{W^{\alpha}\}_{\alpha \in (0,d)}$ be a family of zero-mean Gaussian processes with covariance \eqref{cov} in which $\gamma_0$ is chosen such that its corresponding measure $\mu_0$ satisfies \eqref{cond-mu0}, and the measure $\mu$ is given by:
\begin{equation}
\label{def-mu1}
\mu(d\xi)=|\xi|^{-\alpha}d\xi \quad \mbox{with $\alpha \in (0,d)$ and $d\geq 1$}.
\end{equation}
For any $\alpha \in \big(\max(d-2, 0),d \big)$, let $u^{\alpha}$ be a continuous modification of the Skorohod solution of equation \eqref{pam}, respectively equation \eqref{ham}, with noise $W$ replaced by $W^{\alpha}$.
If $\alpha_n \to \alpha^* \in (\max(d-2,0),d)$, then
\[
u^{\alpha_n} \stackrel{d}{\longrightarrow} u^{\alpha^*} \quad \
\mbox{in $C([0,T] \times \bR^d)$},
\]
where $C([0,T] \times \bR^d)$ is equipped with the topology of uniform convergence on compact sets.
\end{theorem}

\begin{theorem}[The Rough Case]
\label{main-th2}
Let $\{W^{H} \}_{H \in (0,1/2)}$ be a family of zero-mean Gaussian processes with covariance \eqref{cov} in which $\gamma_0$ is given by \eqref{def-gamma0} and the measure $\mu$ is given by:
\begin{equation}
\label{def-mu2}
\mu(d\xi)=|\xi|^{1-2H}d\xi \quad \mbox{with $H \in (0,1/2)$ and $d=1$.}
\end{equation}
Fix $H_0 \in (1/2,1)$. Define
\begin{equation}
\label{def-ell}
\ell=
\left\{
\begin{array}{ll}
\max(3/4-H_0,0) & \mbox{in the case of the heat equation} \\
1/4 & \mbox{in the case of the wave equation}
\end{array} \right.
\end{equation}
For any $H \in (\ell,1/2)$, let $u^{H}$ be a continuous modification of the Skorohod solution of equation \eqref{pam}, respectively \eqref{ham}, with noise $W$ replaced by $W^{H}$.
If $H_n \to H^* \in (\ell,1/2)$, then
\[
u^{H_n} \stackrel{d}{\longrightarrow} u^{H^*} \quad \
\mbox{in $C([0,T] \times \bR)$},
\]
where $C([0,T] \times \bR)$ is equipped with the topology of uniform convergence on compact sets.
\end{theorem}

The proofs of Theorems \ref{main-th1} and \ref{main-th2} are presented in Section \ref{section-regular}, respectively Section \ref{section-rough}, and follow the classical method of finite-dimensional convergence, plus tightness. Instead of the finite-dimensional convergence, we show the (stronger) $L^2(\Omega)$-convergence for every $(t,x)\in \bR_{+} \times \bR^d$ fixed. For this, we first prove the corresponding
$L^2(\Omega)$-convergence of the $m$-th approximation $u_m^{\alpha_n}(t,x)$ (respectively $u_m^{H_n}(t,x)$) of the solution $u^{\alpha^*}(t,x)$ (respectively $u^{H^*}(t,x)$) for fixed $m\geq 1$, and then we let $n\to \infty$, by showing the uniform convergence with respect to the noise parameter ($\alpha$ or $H$). For tightness, we apply Kolmogorov-Centsov theorem, which means that we need to estimate the increments of the solution. This study requires a special analysis of the constants involved, which guarantees that all bounds are uniform for parameter values in a compact set.

\medskip

We conclude this introduction with few words about the notation.
We denote
\[
f(\pmb{t_n},\pmb{x_n})=f(t_1,x_1,\ldots,t_n,x_n),
\]
where $\pmb{t_n}=(t_1,\ldots,t_n)\in \bR_{+}^n$ and $\pmb{x_n}=(x_1,\ldots,x_n)\in (\bR^d)^n$. Occasionally, we will use $\pmb{\xi_n}=(\xi_1,\ldots,\xi_n) \in (\bR^d)^n$. We will use the convention:
\begin{equation}
\label{convention}
G_t(x)=0 \quad \mbox{for any $t<0$ and $x \in \bR^d$.}
\end{equation}

We let $T_n(t)=\{\pmb{t_n}=(t_1,\ldots,t_n); 0<t_1<\ldots<t_n<t\}$ be the $n$-dimensional simplex.

For any $p\geq 1$, we denote by $\|\cdot \|_p$ the norm in $L^p(\Omega)$.
We let $\sigma$ be the surface measure on the unit sphere $S_1(0)=\{z\in \bR^d;|z|=1\}$, and $c_d$ be the area of $S_1(0)$, i.e.
\begin{equation}
\label{def-cd}
c_d=\int_{S_1(0)}\sigma(dz).
\end{equation}

Whenever we need separate calculations for the heat and wave equations involving the fundamental solution $G$, we will use the notation $G^h$ for the heat equation and
$G^w$ for the wave equation.

\section{Regular Noise}
\label{section-regular}

In this section, we consider equations \eqref{pam} and \eqref{ham} driven by a Gaussian noise $W$ with covariance \eqref{cov} in which $\gamma_0$ is chosen such that its corresponding measure $\mu_0$ satisfies \eqref{cond-mu0}, and the measure $\mu$ is given by \eqref{def-mu1}.

To emphasize the dependence on the parameter $\alpha$, we denote the noise, the Hilbert space, and the solution, by $W^{\alpha},\cH^{\alpha},u^{\alpha}$, respectively. The multiple integral of order $n$ with respect to $W^{\alpha}$ will be denoted by $I_n^{\alpha}$.
With this notation, the series expansion \eqref{series} becomes:
\[
u^{\alpha}(t,x)=1+\sum_{n\geq 1}I_{n}^{\alpha}(f_{t,x,n}).
\]

\bigskip

As in \cite{GJQ}, we begin with the construction of the family $\{W^{\alpha}\}_{\alpha \in (0,d)}$ on the same probability space. For this, we consider a $\bC$-valued Gaussian random measure $\widehat{W}=\{\widehat{W}(A); A \in \cB_{b}(\bR^{d+1})\}$ defined on a complete probability space $(\Omega,\cF,\bP)$, given by:
\begin{equation}
\label{def-W-hat}
\widehat{W}(A)=W_1(A)+iW_2(A)
\end{equation}
where $W_1$ and $W_2$ are independent space-time Gaussian white noise processes on $\bR^{d+1}$. Then for any sets $A,B \in \cB_{b}(\bR^{d+1})$,
$\bE[\widehat{W}(A) \overline{\widehat{W}(B)}]={\rm Leb}(A \cap B)$, where ${\rm Leb}$ denotes the Lebesgue measure.

We let $\widehat{W}(1_{A})=\widehat{W}(A)$. By linearity, we extend this definition to simple functions, and by approximation, we extend it further to functions $h \in L_{\bC}^2(\bR^{d+1})$. We denote this extension by
\[
\widehat{W}(h)=\int_{\bR \times \bR^{d}} h(\tau,\xi) \widehat{W}(d\tau,d\xi).
\]

The following isometry property holds: for any function $h \in L_{\bC}^2(\bR^{d+1})$,
\begin{equation}
\label{iso-wide-W}
\bE\left|\int_{\bR \times \bR^d} h(\tau,\xi) \widehat{W}(d\tau,d\xi)\right|^2 =\int_{\bR \times \bR^d}|h(\tau,\xi)|^2 d\tau d\xi.
\end{equation}

Let $\cS_{\bC}(\bR_{+} \times \bR^{d})$ be the set of $\bC$-valued rapidly decreasing $C^{\infty}$-functions on $\bR_{+} \times \bR^{d}$. For any $\varphi \in \cS_{\bC}(\bR_{+} \times \bR^{d})$, we define
\begin{equation}
\label{def-W-alpha}
W^{\alpha}(\varphi)=\int_{\bR \times \bR^{d}} \cF \varphi(\tau,\xi) \sqrt{g_0(\tau)}|\xi|^{-\alpha/2}\widehat{W}(d\tau,d\xi),
\end{equation}
where $\cF \varphi(\tau,\xi)$ is the Fourier transform in both variables:
\[
\cF \varphi(\tau,\xi) = \int_{\bR_{+} \times \bR^{d}}e^{-i \tau t-i \xi \cdot x} \varphi(t,x)dtdx.
\]
Note that $W^{\alpha}(\varphi)$ is real-valued. We use the notation $W^{\alpha}(\varphi)=\int_{0}^{\infty}\int_{\bR^d}
\varphi(t,x)W^{\alpha}(dt,dx)$.

By the isometry property \eqref{iso-wide-W}, it follows that
\[
\bE[W^{\alpha}(\varphi) W^{\alpha}(\psi)]=\int_{\bR \times \bR^{d}}
\cF \varphi(\tau,\xi) \overline{\cF \psi(\tau,\xi) } g_0(\tau)|\xi|^{-\alpha}d\tau d\xi=\langle \varphi,\psi \rangle_{\cH_{\alpha}}.
\]

This shows that the process $W^{\alpha}$ has the desired covariance function \eqref{cov}. Moreover, all processes $(W^{\alpha})_{\alpha \in (0,d)}$ are defined on the same probability space $(\Omega,\cF,\bP)$ as $\widehat{W}$.

\medskip

Although we do not have a complete description of the space $\cH_{\alpha}$, there is a procedure (given by Theorem 2.6 of \cite{BS17}) which allows us to identify some of its elements. We describe this below. More precisely, it can be proved that $\cH_{\alpha}$ contains all functions $t \mapsto \varphi(t,\cdot) \in \cS'(\bR^d)$ for which the Fourier transform $\cF \varphi(t,\cdot)$ is a function for any $t>0$, there exists a version $\phi_{\xi}(t)$ of $\cF \varphi(t,\cdot)(\xi)$ such that
$(t,\xi) \mapsto \phi_{\xi}(t)$ is measurable on $\bR_{+} \times \bR^d$, $\phi_{\xi} \in L^1(\bR_{+})$ for any $\xi \in \bR^d$, and the Fourier transform $\cF \phi_{\xi}$ of $\phi_{\xi}$ satisfies:
\[
\int_{\bR \times \bR^d} |\cF_t \phi_{\xi}(\tau)|^{2} g_0(\tau) |\xi|^{-\alpha}d\xi d\tau<\infty.
\]
For such a function $\varphi$, the stochastic integral with respect to $W^{\alpha}$ can be represented as:
\begin{equation}
\label{def-int}
\int_0^{\infty}\int_{\bR^d}\varphi(t,x)W^{\alpha}(dt,dx)=\int_{\bR \times \bR^d} \cF_t [\cF_x \varphi(t,\cdot)(\xi)](\tau)\sqrt{g_0(\tau)}|\xi|^{-\alpha/2}
\widehat{W}(d\tau,d\xi),
\end{equation}
where $\cF_t,\cF_x$ denote the Fourier transforms in the time and space variables, respectively.
This representation can be extended to multiple integrals with respect to $W^{\alpha}$:
\begin{equation}
\label{In-alpha}
I_k^{\alpha}(\varphi)=\int_{(\bR \times \bR^d)^k} \cF_t[\cF_x \varphi(\pmb{t_k},\bullet)(\pmb{\xi_k})](\pmb{\tau_k})
\prod_{j=1}^{k}\sqrt{g_0(\tau_j)}|\xi_j|^{-\alpha/2}
\widehat{W}(d\tau_1,d\xi_1) \ldots \widehat{W}(d\tau_k,d\xi_k),
\end{equation}
where $\cF_x$ is the Fourier transform in the space variables and $\cF_t $ is the Fourier transform in the time variables, i.e. the Fourier transform of the function $\pmb{t_k} \mapsto \cF_x \varphi(\pmb{t_k},\bullet)(\pmb{\xi_k})$. Whenever there is no risk of confusion, we drop the lower indices $t,x$ from the Fourier transform notation.

\medskip

The following two lemmas will play an important role below.
The first one can be found for instance in \cite{B12-POTA}, while the second one is proved by elementary methods.

\begin{lemma}
\label{G-lemma}
For any $\alpha \in \big(\max(d-2,0),d\big)$ and $t>0$,
\begin{align*}
& \int_{\bR^d}|\cF G_t(\xi)|^2 |\xi-\eta|^{-\alpha}d\xi \leq K_{d,\alpha} \,t^{r_{\alpha}},
\end{align*}
where
\begin{equation}
\label{def-K}
K_{d,\alpha}=\int_{\bR^d}\frac{1}{1+|\xi|^{2}}|\xi|^{-\alpha}d\xi \leq c_{d}\left( \frac{1}{d-\alpha}+\frac{1}{2-(d-\alpha)}\right),
\end{equation}
with $c_d$ is given by \eqref{def-cd} and
\begin{equation}
\label{def-r}
r_{\alpha}=
\left\{
\begin{array}{ll}
-(d-\alpha)/2 & \mbox{for the heat equation} \\
2-(d-\alpha) & \mbox{for the wave equation}
\end{array} \right.
\end{equation}
\end{lemma}

\begin{lemma}
\label{beta-lem}
For any $\beta_1,\ldots,\beta_n>-1$,
\[
\int_{T_n(t)}\prod_{j=1}^{n}(t_{j+1}-t_j)^{\beta_j} d\pmb{t_n}=\frac{\prod_{j=1}^n \Gamma(\beta_j+1)\,t^{|\beta|+n}}{\Gamma(|\beta|+n+1)},
\]
where $|\beta|=\sum_{j=1}^n \beta_j$ and $t_{n+1}=t$.
\end{lemma}

The following result will be used in the proof of Theorem \ref{main-th1}, in order to show the finite dimensional convergence.

\begin{lemma}
\label{Ik-conv}
If $\alpha_n \to \alpha^* \in \big(\max(d-2,0),d\big)$, then for any $t>0$, $x \in \bR^d$ and $k\geq 1$,
\[
\bE|I_k^{\alpha_n}(f_{t,x,k}) - I_k^{\alpha^*}(f_{t,x,k})|^2 \to 0, \quad \mbox{as} \quad n\to \infty.
\]
\end{lemma}

\begin{proof} By \eqref{In-alpha},
\[
I_k^{\alpha_n}(f_{t,x,k}) - I_k^{\alpha^*}(f_{t,x,k})
=\int_{(\bR \times \bR^d)^n} H_k^{(n)}(\pmb{\tau_k},\pmb{\xi_k})
\widehat{W}(d\tau_1,d\xi_1) \ldots \widehat{W}(d\tau_n,d\xi_n),
\]
where
\[
H_k^{(n)}(\pmb{\tau_k},\pmb{\xi_k})=\cF \phi_{\pmb{\xi_k}}(\pmb{\tau_k})
\prod_{j=1}^{k}\sqrt{g_0(\tau_j)}
\left(\prod_{j=1}^k|\xi_j|^{-\alpha_n/2}-
\prod_{j=1}^k|\xi_j|^{-\alpha^*/2}\right)
\]
and
\begin{equation}
\label{def-phi-k}
\phi_{\pmb{\xi_k}}(\pmb{t_k}):=\cF f_{t,x,k}(\pmb{t_k},\cdot)(\pmb{\xi_k})=e^{-i(\xi_1+\ldots+\xi_k)\cdot x}
\prod_{j=1}^k \cF G_{t_{j+1}-t_j}(\xi_1+\ldots+\xi_j)
\end{equation}
with $t_{k+1}=t$. Here, we use convention \eqref{convention}. Hence,
\[
Q_n :=\bE|I_k^{\alpha_n}(f_{t,x,k}) -I_k^{\alpha^*}(f_{t,x,k})|^2=k! \| \widetilde{H}_k^{(n)}\|_{L_{\bC}^2((\bR \times \bR^d)^k)}^2 \leq k! \| H_k^{(n)}\|_{L_{\bC}^2((\bR \times \bR^d)^k)}^2
\]
where $\widetilde{H}_k^{(n)}$ is the symmetrization of $H_k^{(n)}$. Using \eqref{phi-high-dim}, we have:
\begin{align}
\nonumber
Q_n & \leq k! \int_{(\bR^d)^k} \left|\prod_{j=1}^k|\xi_j|^{-\alpha_n/2}-
\prod_{j=1}^k|\xi_j|^{-\alpha^*/2}\right|^2 \int_{\bR^k}  |\cF \phi_{\pmb{\xi_k}}(\pmb{\tau_k})|^2 \prod_{j=1}^k g_0(\tau_j)
d\pmb{\tau_k} d\pmb{\xi_k}\\
\label{est-Qn}
&= k! \int_{T_k(t)}\int_{T_k(t)} \prod_{j=1}^k \gamma_0(t_j-s_j)A_k^{(n)}(\pmb{t_k},\pmb{s_k})d\pmb{t_k}d\pmb{s_k}
\end{align}
where $T_k(t)=\{0<t_1<\ldots<t_k<t\}$ and
\begin{align*}
A_k^{(n)}(\pmb{t_k},\pmb{s_k})=\int_{(\bR^d)^k}
\phi_{\pmb{\xi_k}}(\pmb{t_k})
\phi_{\pmb{\xi_k}}(\pmb{s_k}) \left|\prod_{j=1}^k|\xi_j|^{-\alpha_n/2}-
\prod_{j=1}^k|\xi_j|^{-\alpha^*/2}\right|^2 d\pmb{\xi_k}.
\end{align*}
By Cauchy-Schwarz inequality and inequality $ab \leq \frac{1}{2}(a^2+b^2)$, we have:
\begin{equation}
\label{A-ineq}
A_k^{(n)}(\pmb{t_k},\pmb{s_k}) \leq A_k^{(n)}(\pmb{t_k},\pmb{t_k})^{1/2} A_k^{(n)}(\pmb{s_k},\pmb{s_k})^{1/2} \leq \frac{1}{2} \big( A_k^{(n)}(\pmb{t_k},\pmb{t_k})+A_k^{(n)}(\pmb{s_k},\pmb{s_k}) \big).
\end{equation}
Hence,
\[
Q_n \leq k! \int_{T_k(t)} \int_{T_k(t)} A_k^{(n)}(\pmb{t_k},\pmb{t_k})
\left(\prod_{j=1}^k \gamma_0(t_j-s_j)
d \pmb{s_k}\right) d \pmb{t_k} \leq k! \Gamma_{0,t}^k\int_{T_k(t)} A_k^{(n)}(\pmb{t_k},\pmb{t_k}) d \pmb{t_k},
\]
where $\Gamma_{0,t}=\int_{-t}^t \gamma_0(s)ds$. It suffices to show that the integral appearing in the upper bound above converges to $0$ as $n\to \infty$, i.e.
\[
T_n:=\int_{T_k(t)} \int_{(\bR^d)^k} \prod_{j=1}^k \big|\cF G_{t_{j+1}-t_j}(\xi_1+\ldots+\xi_j) \big|^2 \left|\prod_{j=1}^k|\xi_j|^{-\alpha_n/2}-
\prod_{j=1}^k|\xi_j|^{-\alpha^*/2}\right|^2
  d\pmb{\xi_k} d \pmb{t_k} \to 0.
\]

Note that the integrand converges pointwise  to $0$ on $T_k(t) \times (\bR^d)^k$, as $n \to \infty$.
To apply the Dominated Convergence Theorem, we need to bound the integrand by an integrable function. In fact, it suffices to find an integrable bound for the term depending on $\alpha_n$:
\[
g_n(\pmb{t_k},\pmb{\xi_k}):=\prod_{j=1}^k \big|\cF G_{t_{j+1}-t_j}(\xi_1+\ldots+\xi_j) \big|^2 \prod_{j=1}^k|\xi_j|^{-\alpha_n}.
\]
We have to find a function $g$ on $T_k(t) \times (\bR^d)^k$, such that
\[
g_n(\pmb{t_k},\pmb{\xi_k}) \leq g(\pmb{t_k},\pmb{\xi_k}) \quad \mbox{for all $n$}, \quad \mbox{and} \quad \int_{T_k(t)} \int_{(\bR^d)^k}g(\pmb{t_k},\pmb{\xi_k}) d\pmb{\xi_k} d\pmb{t_k}<\infty.
\]

Recall that $d-2 < \alpha^{\ast} < d$. Fix numbers $a$ and $b$ such that $d-2 < a < \alpha^{\ast} < b < d$. Since $\alpha_n \to \alpha^{\ast}$, there exists $N \in \mathbb{N}$ such that
$a \le \alpha_n \le b$ for all $n \ge N$.
We take:
\[
g(\pmb{t}_k,\pmb{\xi_k}) := \prod_{j=1}^k \big|\cF G_{t_{j+1}-t_j}(\xi_1+\ldots+\xi_j)\big|^2
 \big(|\xi_j|^{-b}1_{\{ |\xi_j| \le 1 \}}  + |\xi_j|^{-a}1_{\{|\xi_j| > 1 \}}\big).
\]
Note that, for any $j=1,\ldots,k$,
\begin{align*}
& \int_{\bR^d} |\cF G_{t_{j+1}-t_j}(\xi_1+\ldots+\xi_j)|^2
 \big(|\xi_j|^{-b}1_{\{ |\xi_j| \le 1 \}}  + |\xi_j|^{-a}1_{\{|\xi_j| > 1 \}}\big) d\xi_j \\
 &\leq C_t\int_{|\xi_j|\leq 1}|\xi_j|^{-b}d\xi_j+\int_{\bR^d}
|\cF G_{t_{j+1}-t_j}(\xi_1+\ldots+\xi_j)|^2 |\xi_j|^{-a}d\xi_j,
\end{align*}
where $C_t=1$ for the heat equation, and $C_t=t^2$ for the wave equation.
By Lemma \ref{G-lemma}, the last integral is bounded by $K_{d,a}(t_{j+1}-t_j)^{r_{a}}$, where $K_{d,a}$ and $r_{a}$ are given by \eqref{def-K}, respectively \eqref{def-r}. Let
$C_{a,b,d}=\frac{c_d}{d-b}+K_{d,a}$, where $c_d$ is the area of $\{z\in \bR^d;|z|=1\}$. Then,
\[
\int_{T_k(t)} \int_{(\bR^d)^k}g(\pmb{t_k},\pmb{x_k}) d\pmb{t_k} d\pmb{\xi_k} \leq C_{a,b,d}^k\int_{T_k(t)}\prod_{j=1}^{k}
\Big(C_t+(t_{j+1}-t_j)^{r_{a}}\Big) d\pmb{t_k},
\]
and the last integral is finite due to Lemma \ref{beta-lem}, since $r_a>-1$.
\end{proof}

We now start the preparations for the proof of tightness.
We begin by recalling some path properties of the solution $u^{\alpha}$.
The next result is a restatement Theorem 3.2 of \cite{BQS} (for heat equation), respectively Theorem 8.3 of \cite{BS17} (for the wave equation), in the case when the spatial covariance kernel of the noise is the Riesz kernel $f(x)=C_{d,\alpha}|x|^{-(d-\alpha)}$. We include this result for the sake of completeness, but we cannot use it directly for the proof of tightness, since we need the explicit form of the constant $C$ as a function of $\alpha$, to be able to bound it uniformly for all $\alpha$ in a compact interval $[a,b]$.

\begin{theorem}
\label{Holder-th}
Assume that $\alpha \in \big(\max(d-2,0),d\big)$.
Let $p\geq 2$ and $T>0$ be arbitrary, and $K \subset \bR^d$ be a compact set.
For any $\beta \in (\frac{d-\alpha}{2},1)$, $t,t' \in [0,T]$, and $x,x' \in K$, we have
\begin{align}
\label{Holder-eq}
& \|u^{\alpha}(t,x)-u^{\alpha}(t',x')\|_p \leq C\big(|t-t'|^{\gamma} +|x-x'|^{1-\beta}\big),
\end{align}
where $\gamma=\frac{1-\beta}{2}$ for the heat equation, and $\gamma=1-\beta$ for the wave equation. Here
$C$ is a constant that depends on $(\alpha,d,p,T,K,\beta)$.
\end{theorem}

By Theorem \ref{Holder-th}, the process $u^{\alpha}$ has a continuous modification in $(t,x)$. We work with this modification, which we denote also by $u^{\alpha}$.

In the case of the heat equation, we will use the following constant:
\[
k_{\alpha}(t):=\int_{\bR^d} e^{-t|\xi|^2} |\xi|^{-\alpha}d\xi=\frac{c_d}{2}\Gamma\left(\frac{d-\alpha}{2}\right) \, t^{-\frac{d-\alpha}{2}},
\]
with $c_d$ given by \eqref{def-cd}. We will need the following result.

\begin{lemma}
\label{G-lem2}
Let $\max(d-2,0)<a<b<d$ and $\beta \in (\frac{d-a}{2},1)$ be arbitrary. For any $\alpha \in [a,b]$,

(a) for any $t\in [0,T]$,
\[
\sup_{\eta \in \bR^d} \int_{\bR^d} \big|\cF G_{t}(\xi+\eta)\big|^2 |\xi|^{-\alpha}d\xi \leq
\left\{
\begin{array}{ll}
k_{\alpha}(t) & \mbox{for the heat equation} \\
C_{T,d,a,b,\beta} t^{2(1-\beta)} & \mbox{for the wave equation}
\end{array} \right.
\]

(b) for any $t>0$,
\[
\sup_{\eta \in \bR^d} \int_{\bR^d} \big|\cF \big(G_{t+h}-G_t\big)(\xi+\eta)\big|^2 |\xi|^{-\alpha}d\xi \leq
\left\{
\begin{array}{ll}
C_{\beta} h^{1-\beta} t^{\beta-1}k_{\alpha}(t/2) & \mbox{for the heat equation} \\
C_{d,a,b,\beta} h^{2(1-\beta)} & \mbox{for the wave equation}
\end{array} \right.
\]

(c) for any compact set $K \subset \bR^d$ and for any $z \in K$,
\[
\sup_{\eta \in \bR^d} \int_{\bR^d} |e^{i(\xi+\eta)\cdot z}-1|^2
|\cF G_t(\xi+\eta)|^2 |\xi|^{-\alpha}d\xi \leq
\left\{
\begin{array}{ll}
C_{\beta}|z|^{2(1-\beta)}t^{\beta-1}k_{\alpha}(t/2) & \mbox{for the heat equation} \\
C_{T,K,d,a,b,\beta} |z|^{2(1-\beta)} & \mbox{for the wave equation}
\end{array} \right.
\]
\end{lemma}

\begin{proof}
This is result is proved exactly as Proposition 3.1 of \cite{BQS} for the heat equation, respectively Proposition 7.4 of \cite{CD08} for the wave equation, except that here we give a uniform bound for the constants, for all $\alpha \in [a,b]$ (which can be obtained after examining closely the proofs).
In these references, the results are proved for {\em fixed} $\alpha \in \big(\max(d-2,0),d\big)$ and the exponents are given in terms of an arbitrary $\beta>0$ which satisfies the condition:
\begin{equation}
\label{beta-cond}
\int_{\bR^d}\left(\frac{1}{1+|\xi|^2}\right)^{\beta}
|\xi|^{-\alpha}d\xi<\infty,
\end{equation}
which is equivalent to the requirement $\beta \in \left(\frac{d-\alpha}{2},1\right)$.
In our case, since we want to obtain a {\em uniform} bound for all $\alpha \in [a,b]$, we fix an apriori value $\beta \in \left( \frac{d-a}{2},1\right)$, which guarantees that condition \eqref{beta-cond} holds for all $\alpha \in [a,b]$.
\end{proof}

The next result gives the desired upper bound for the increments of the solution (uniformly in $\alpha \in [a,b]$), which will be used below to show the tightness of $(u^{\alpha})_{\alpha \in [a,b]}$.

\begin{theorem}
\label{Th-uniform-C}
Assume that $\max(d-2,0)<a<b<d$. Let $p\geq 2$, $T>0$ be arbitrary and $K \subset \bR^d$ be a compact set. For any $\beta \in \big(\frac{d-a}{2},1 \big)$, $t,t' \in [0,T]$ and $x,x' \in K$, we have:
\[
\sup_{\alpha \in [a,b]} \|u^{\alpha}(t,x)-u^{\alpha}(t',x')\|_p \leq C \Big(|t-t'|^{\gamma}+|x-x'|^{1-\beta}\Big),
\]
where $\gamma=\frac{1-\beta}{2}$ for the heat equation, respectively $\gamma=1-\beta$ for the wave equation.
Here $C$ is a constant that depends on $(d,p,T,K,\beta)$.
\end{theorem}

\begin{proof}
We follow the same argument as in the proof of Theorem 3.2 of \cite{BQS} for the heat equation, respectively Theorem 8.3 of \cite{BS17} for the wave equation.

{\bf Step 1.} We study the time increments. Let $t \in [0,T]$ and $h>0$ be such that $t+h \in [0,T]$. By hypercontractivity,
\begin{align*}
\big\|u^{\alpha}(t+h,x)-u^{\alpha}(t,x)\big\|_p & \leq \sum_{n\geq 1}(p-1)^{n/2}\big\|I_n^{\alpha}\big( f_{t+h,x,n}-f_{t,x,n}\big)\big\|_2 \\
& \leq \sum_{n\geq 1}(p-1)^{n/2}
\left(\frac{2}{n!} \Big(A_n^{\alpha}(t,h)+B_n^{\alpha}(t,h)\Big) \right)^{1/2},
\end{align*}
where
$A_n^{\alpha}(t,h)=(n!)^2\| \widetilde{f}_{t+h,x,n}1_{[0,t]^n}-\widetilde{f}_{t,x,n}
\|_{\cH_{\alpha}^{\otimes n}}^2$ and $B_n^{\alpha}(t,h)=(n!)^2\| \widetilde{f}_{t+h,x,n}1_{[0,t+h]^n -[0,t]^n}\|_{\cH_{\alpha}^{\otimes n}}^2$.

We examine $A_n^{\alpha}(t,h)$. Note that
\[
A_n^{\alpha}(t,h) \leq \Gamma_{0,t}^n \int_{[0,t]^n} \psi_{t,h,n}^{\alpha}(\pmb{t_n})d \pmb{t_n},
\]
where the function $\psi_{t,h,n}^{\alpha}(\pmb{t_n})$ is defined as follows: if $\pmb{t_n} \in [0,t]^n$ and $\rho$ is the permutation of $1,\ldots,n$ such that $t_{\rho(1)}<\ldots<t_{\rho(n)}$, we let $u_j=t_{\rho(j+1)}-t_{\rho(j)}$ (with $t_{\rho(n+1)}=t$), and
\begin{align*}
\psi_{t,h,n}^{\alpha}(\pmb{t_n}) & : =
\int_{(\bR^d)^n} \prod_{j=1}^{n-1}\big|\cF G_{u_j}(\xi_1+\ldots+\xi_j)\big|^2 \big|\cF (G_{u_n+h}-G_{u_n})(\xi_1+\ldots+\xi_n)\big|^2 \prod_{j=1}^{n}
|\xi_j|^{-\alpha}d\pmb{\xi_n}.
\end{align*}

To find an upper bound for $\psi_{t,h,n}^{\alpha}(\pmb{t_n})$, we use Lemma \ref{G-lem2}.(a,b). For the heat equation,
\begin{equation}
\label{bound-psi-h}
\psi_{t,h,n}^{\alpha}(\pmb{t_n}) \leq C_{\beta}
h^{1-\beta} \prod_{j=1}^{n-1}k_{\alpha}(u_j) u_n^{\beta-1} k_{\alpha}(u_n/2),
\end{equation}
and hence, using the form of $k_{\alpha}$ and Lemma \ref{beta-lem}, we have:
\begin{align*}
A_n^{\alpha}(t,h) & \leq C_{\beta} h^{1-\beta} n!  \, \Gamma_{0,t}^n c_d^n \Gamma\left(\frac{d-\alpha}{2}\right)^n \int_{T_n(t)} \prod_{j=1}^{n-1}(t_{j+1}-t_j)^{\frac{\alpha-d}{2}}
(t-t_n)^{\beta-1-\frac{d-\alpha}{2}}d\pmb{t_n}\\
& \leq C_{\beta} h^{1-\beta} n!  \, \Gamma_{0,t}^n c_d^n \Gamma\left(\frac{d-\alpha}{2}\right)^n \frac{\Gamma(1-\frac{d-\alpha}{2})^{n-1}\,t^{(n-1)(1-\frac{d-\alpha}{2})}}{
\Gamma\big((n-1)(1-\frac{d-\alpha}{2})+1\big)} \cdot
\frac{t^{\beta-\frac{d-\alpha}{2}}}{\beta-\frac{d-\alpha}{2}}.
\end{align*}
This expression can be bounded uniformly in $\alpha \in [a,b]$, using monotonicity properties of the $\Gamma$ function. For the wave equation,
\begin{equation}
\label{bound-psi-w}
\psi_{t,h,n}^{\alpha}(\pmb{t_n}) \leq C^n h^{2(1-\beta)} \prod_{j=1}^{n-1}u_j^{2(1-\beta)}
\end{equation}
where $C$ is a constant that depends on $T,d,a,b,\beta$, and hence
\begin{align*}
A_n^{\alpha}(t,h) & \leq C^n h^{2(1-\beta)} n!  \, \Gamma_{0,t}^n \int_{T_n(t)} \prod_{j=1}^{n-1}(t_{j+1}-t_j)^{2(1-\beta)}
d\pmb{t_n}\\
&\leq C^n h^{2(1-\beta)} n!  \, \Gamma_{0,t}^n \frac{\Gamma(3-2\beta)^{n-1}}{\Gamma\big((n-1)(3-2\beta)+1\big)}
t^{(n-1)(3-2\beta)+1}.
\end{align*}
In summary, we obtain that:
\begin{align*}
\sup_{\alpha \in [a,b]}\sum_{n\geq 1} (p-1)^{n/2}\left(\frac{1}{n!}A_{n}^{\alpha}(t,h)\right)^{1/2}\leq
\left\{
\begin{array}{ll}
C h^{\frac{1-\beta}{2}}  & \mbox{for the heat equation} \\
 C h^{1-\beta} & \mbox{for the wave equation}
\end{array} \right.
\end{align*}
where $C$ is a constant that depends on $(p,d,a,b,T,\beta)$.

Next, we examine $B_n^{\alpha}(t,h)$. Let $D_{t,h}=[0,t+h]^n \verb2\2 [0,t]^n$. Then
\begin{align*}
B_n^{\alpha}(t,h) & \leq \Gamma_{0,t+h}^n \int_{[0,t+h]^n} \gamma_{t,h,n}^{\alpha}(\pmb{t_n})1_{D_{t,h}}(\pmb{t_n}) d\pmb{t_n},
\end{align*}
where
\[
\gamma_{t,h,n}^{\alpha}(\pmb{t_n}):=\int_{(\bR^d)^n} \prod_{j=1}^{n-1}\big|\cF G_{u_j}(\xi_1+\ldots+\xi_j)\big|^2
\big|\cF G_{u_n+h}(\xi_1+\ldots+\xi_n)\big|^2 \prod_{j=1}^{n}|\xi_j|^{-\alpha}d\pmb{\xi_n},
\]
and we define, as above, $u_j=t_{\rho(j+1)}-t_{\rho(j)}$ if $t_{\rho(1)}<\ldots<t_{\rho(n)}<t+h$ and $t_{\rho(n+1)}=t$.

To find an upper bound for $\gamma_{t,h,n}^{\alpha}(\pmb{t_n})$, we use Lemma \ref{G-lem2}.(a). For the heat equation,
\[
\gamma_{t,h,n}^{\alpha}(\pmb{t_n}) \leq \prod_{j=1}^{n-1}k_{\alpha}(u_j) k_{\alpha}(u_n+h),
\]
and hence, using the form of $k_{\alpha}$ and Lemma \ref{beta-lem}, we have:
\begin{align*}
B_n^{\alpha}(t,h)  & \leq  n! \, \Gamma_{0,T}^n c_d^n \Gamma\Big(\frac{d-\alpha}{2}\Big)^n \int_t^{t+h} \int_{T_{n-1}(t_n)} \prod_{j=1}^{n-1}(t_{j+1}-t_j)^{-\frac{d-\alpha}{2}} (t+h-t_n)^{-\frac{d-\alpha}{2}} d\pmb{t_{n-1}}dt_n\\
& \leq n! \, \Gamma_{0,T}^n c_d^n \Gamma\Big(\frac{d-\alpha}{2}\Big)^n \frac{\Gamma(1-\frac{d-\alpha}{2})^{n-1} \, t^{(n-1)(1-\frac{d-\alpha}{2})}}{\Gamma
\big((n-1)(1-\frac{d-\alpha}{2})+1\big)} \cdot
\frac{h^{1-\frac{d-\alpha}{2}}}{1-\frac{d-\alpha}{2}}.
\end{align*}
It is clear that the previous expression can be bounded uniformly in $\alpha \in [a,b]$. We also note that $h^{1-\frac{d-\alpha}{2}} \leq C_T h^{1-\beta}$ since $\beta>\frac{d-a}{2}>\frac{d-\alpha}{2}$.
For the wave equation,
\[
\gamma_{t,h,n}^{\alpha}(\pmb{t_n}) \leq C^n \prod_{j=1}^{n-1}u_j^{2(1-\beta)} (u_n+h)^{2(1-\beta)},
\]
where $C$ is a constant that depends on $(T,d,a,b,\beta)$, and hence
\begin{align*}
B_{n}^{\alpha}(t,h) & \leq n! \Gamma_{0,T}^n C^n \int_t^{t+h} \int_{T_{n-1}(t_n)} \prod_{j=1}^{n-1}(t_{j+1}-t_j)^{2(1-\beta)} (t+h-t_n)^{2(1-\beta)} d\pmb{t_{n-1}}dt_n\\
&\leq  n! \Gamma_{0,T}^n C^n \frac{\Gamma(3-2\beta)^{n-1} \, t^{(n-1)(3-2\beta)}}{\Gamma\big((n-1)(3-2\beta)+1\big)}
\cdot \frac{h^{3-2\beta}}{3-2\beta}.
\end{align*}
In summary, we have:
\begin{align*}
\sup_{\alpha \in [a,b]}\sum_{n\geq 1} (p-1)^{n/2}\left(\frac{1}{n!}B_{n}^{\alpha}(t,h)\right)^{1/2}\leq
\left\{
\begin{array}{ll}
C h^{\frac{1-\beta}{2}}  & \mbox{for the heat equation} \\
C h^{(3-2\beta)/2} & \mbox{for the wave equation}
\end{array} \right.
\end{align*}

\medskip

{\bf Step 2.} We study the space increments. By hypercontractivity,
\[
\|u^{\alpha}(t,x+z)-u^{\alpha}(t,x)\|_p \leq \sum_{n\geq 1}(p-1)^{n/2} \|I_{n}^{\alpha}(f_{t,x+z,n}-f_{t,x,n})\|_p
\leq \sum_{n\geq 1}(p-1)^{n/2} \left( \frac{1}{n!} C_{n}^{\alpha}(t,z)\right)^{1/2},
\]
where $C_{n}^{\alpha}(t,z)=(n!)^2 \| \widetilde{f}_{t,x+z,n}-\widetilde{f}_{t,x,n}
\|_{\cH_{\alpha}^{\otimes n}}^{2}$. Note that
\begin{align*}
C_{n}^{\alpha}(t,z) & \leq \Gamma_{0,t}^n \int_{[0,t]^n}\psi_{t,z,n}^{\alpha}(\pmb{t_n}) d\pmb{t_n},
\end{align*}
where
\[
\psi_{t,z,n}^{\alpha}(\pmb{t_n}) =\int_{(\bR^d)^n}
\prod_{j=1}^{n}|\cF G_{u_j}(\xi_1+\ldots+\xi_j)|^2 |1-e^{-i (\xi_1+\ldots+\xi_n)\cdot z}|^2 \prod_{j=1}^{n}|\xi_j|^{-\alpha}d\pmb{\xi_n},
\]
and we define, as above, $u_j=t_{\rho(j+1)}-t_{\rho(j)}$ if $t_{\rho(1)}<\ldots<t_{\rho(n)}<t+h$ and $t_{\rho(n+1)}=t$.

To find an upper bound for $\psi_{t,z,n}^{(\alpha)}(\pmb{t_n})$, we use Lemma \ref{G-lem2}.(a,c). For the heat equation,
\begin{align*}
\psi_{t,z,n}^{\alpha}(\pmb{t_n}) \leq C_{\beta} |z|^{2(1-\beta)} \prod_{j=1}^{n-1} k_{\alpha}(u_j) u_n^{\beta-1} k_{\alpha}(u_n/2),
\end{align*}
which is similar to \eqref{bound-psi-h} above. For the wave equation,
\begin{align*}
\psi_{t,z,n}^{\alpha}(\pmb{t_n}) \leq C^n |z|^{2(1-\beta)}\prod_{j=1}^{n-1}u_j^{2(1-\beta)},
\end{align*}
with a constant $C>0$ that depends on $(T,K,d,a,b)$, which is similar to \eqref{bound-psi-w} above. We conclude that for both heat and wave equations,
\begin{align*}
\sup_{\alpha \in [a,b]}\sum_{n\geq 1} (p-1)^{n/2}\left(\frac{1}{n!}C_{n}^{\alpha}(t,h)\right)^{1/2}\leq
C |z|^{1-\beta}.
\end{align*}
\end{proof}

\medskip

{\bf Proof of Theorem \ref{main-th1}}:

 {\em Step 1. (finite-dimensional convergence)} In this step, we prove that:
\[
\big(u^{\alpha_n}(t_1,x_1),\ldots,u^{\alpha_n}(t_k,x_k)\big)
\stackrel{d}{\to}
\big(u^{\alpha^*}(t_1,x_1),\ldots,u^{\alpha^*}(t_k,x_k)\big),
\]
for any $(t_1,x_1),\ldots,
(t_k,x_k)\in [0,T] \times \bR^d$. It is enough to prove that for any $(t,x) \in [0,T] \times \bR^d$, $u^{\alpha_n}(t,x) \to u^{\alpha^*}(t,x)$ in $L^2(\Omega)$ as $n\to \infty$.
To do this, we approximate $u^{\alpha}(t,x)$ by the partial sum:
\[
u_m^{\alpha}(t,x)=1+\sum_{k=1}^{m}I_k^{\alpha}(f_{t,x,k}).
\]
Note that, in the case of the white noise in time considered in \cite{GJQ},
$(u_m)_{m\geq 0}$ coincides with the Picard iteration sequence.
Our argument is summarized by the diagram below:
\begin{equation*}
\begin{tikzcd}[column sep=large, row sep=large,remember picture]
u^{\alpha_n}_m(t,x) \arrow[r, "n \to \infty", "\forall \, m"',] \arrow[d, "m \to \infty", "\text{uniformly in } n"'] & u^{\alpha^\ast}_m(t,x) \arrow[d, "m \to \infty"]\\
u^{\alpha_n}(t,x) \arrow[r, dashrightarrow] & u^{\alpha^*}(t,x)
\end{tikzcd}
\end{equation*}


The convergence on the top line of the diagram holds by Lemma \ref{Ik-conv}: for any $m\geq 1$,
\begin{align*}
\bE|u_m^{\alpha_n}(t,x)-u_m^{\alpha^*}(t,x)|^2 \leq m\sum_{k=1}^{m}
\bE|I_{k}^{\alpha_n}(f_{t,x,k})-I_{k}^{\alpha^*}(f_{t,x,k})|^2 \quad \mbox{as $n\to \infty$}.
\end{align*}
The convergence indicated by the right vertical arrow holds trivially since by construction,
$u_m^{\alpha}(t,x) \to u^{\alpha}(t,x)$ in $L^2(\Omega)$ as $m \to \infty$, for any $\alpha \in \big(\max(d-2,0),d\big)$.

It remains to prove the uniform convergence in $n$. For this, we
fix some values $a$ and $b$ such that $\max(d-2,0)<a<\alpha^*<b<d$. Since $\alpha_n \to \alpha^*$, there exists $N \in \bN$ such that $a<\alpha_n<b$ for any $n\geq N$. Therefore, it suffices to prove that:
\begin{equation}
\label{uniform}
\sup_{\alpha \in [a,b]}\bE|u_m^{\alpha}(t,x)-u^{\alpha}(t,x)|^2
=\sup_{\alpha \in [a,b]} \sum_{k\geq m+1}\bE|I_k^{\alpha}(f_{t,x,k})|^2
\to 0 \quad \mbox{as $m \to \infty$.}
\end{equation}

Note that
\[
\bE|I_k^{\alpha}(f_{t,x,k})|^2 =k! \|\widetilde{f}_{t,x,k} \|_{\cH_{\alpha}^{\otimes k}}^{2} \leq  k! \Gamma_{0,t}^k
\int_{[0,t]^n}A_k^{\alpha}(\pmb{t_k})d\pmb{t_k},
\]
where $\Gamma_{0,t}=\int_{-t}^t \gamma_0(s)ds$ and
$
A_{k}^{\alpha}(\pmb{t_k})=\int_{(\bR^d)^k}
\big|\cF \widetilde{f}_{t,x,k}(\pmb{t_k},\bullet)(\pmb{\xi_k})\big|^2 \prod_{j=1}^k |\xi_j|^{-\alpha}d\pmb{\xi_k}$.

If $\pmb{t_k}=(t_1,\ldots,t_k) \in [0,t]^k$ is arbitrary and  $\rho$ is a permutation of $1,\ldots,k$ such that $t_{\rho(1)}<\ldots<t_{\rho(k)}$, with $t_{\rho(k+1)}=t$, then by Lemma \ref{G-lemma}, we have:
\begin{align*}
A_{k}^{\alpha}(\pmb{t_k})&=\frac{1}{(k!)^2}
\int_{(\bR^d)^k} \prod_{j=1}^{k}\big|\cF G_{t_{\rho(j+1)}-t_{\rho(j)}}\big(\sum_{i=1}^j \xi_i\big)\big|^2 \,|\xi_j|^{-\alpha}d\pmb{\xi_k} \leq \frac{K_{d,\alpha}^k }{(k!)^2} \prod_{j=1}^{k}(t_{\rho(j+1)}-t_{\rho(j)})^{r_{\alpha}},
\end{align*}
where $K_{d,\alpha}$ and $r_{\alpha}$ are given by \eqref{def-K}, respectively \eqref{def-r}.
Using Lemma \ref{beta-lem}, we obtain:
\begin{align*}
\bE|I_k^{\alpha}(f_{t,x,k})|^2 & \leq \Gamma_{0,t}^k K_{d,\alpha}^k  \int_{T_k(t)} \prod_{j=1}^{k}(t_{j+1}-t_j)^{r_{\alpha}} d\pmb{t_k}=\Gamma_{0,t}^k K_{d,\alpha}^k \frac{\Gamma(r_{\alpha}+1)^k \, t^{k(r_{\alpha}+1)}}{\Gamma\big(k(r_{\alpha}+1)+1\big)}.
\end{align*}
Our next task is to find a bound for the last expression, uniformly in $\alpha \in [a,b]$.

First, note that $K_{d,\alpha}\leq c_d \left(\frac{1}{d-b} +\frac{1}{2-(d-a)} \right)=:K_{d,a,b}$. Moreover,
recall that there exists a value $x_0 \in (1,2)$ such that $\Gamma$ is decreasing on $(0,x_0)$ and increasing on $(x_0,\infty)$.

In the case of the heat equation, $r_{\alpha}=-\frac{d-\alpha}{2}$, for any $\alpha \in [a,b]$, $\Gamma(r_{\alpha}+1)\leq \Gamma(1-\frac{d-a}{2})$, $t^{k(r_{\alpha}+1)} \leq (t \vee 1)^{k(1-\frac{d-b}{2})}$,
$$ \Gamma\big(k(r_{\alpha}+1)+1\big) \geq \Gamma\Big(k\Big(1-\frac{d-a}{2}\Big)+1\Big) \quad \mbox{for any} \ k \geq m_0,$$ where $m_0$ is chosen such that $m_0(1-\frac{d-a}{2})>x_0$, and therefore for any $m\geq m_0$,
\begin{equation}
\label{sum-m-heat}
\sup_{\alpha \in [a,b]}\sum_{k\geq m+1}\bE|I_k^{\alpha}(f_{t,x,k})|^2 \leq \sum_{k\geq m+1}
\Gamma_{0,t}^k K_{d,a,b}^k \frac{\Gamma(1-\frac{d-a}{2})^k \, (t \vee 1)^{k(1-\frac{d-b}{2})}}{\Gamma\Big(k\Big(1-\frac{d-a}{2}\Big)+1\Big)}.
\end{equation}

In the case of the wave equation, $r_{\alpha}=2-d+\alpha$, for any $\alpha \in [a,b]$, $\Gamma(r_{\alpha}+1)\leq \Gamma(3)=2$, $t^{k(r_{\alpha}+1)} \leq (t \vee 1)^{k(3-d+b)}$,
$$ \Gamma\big(k(r_{\alpha}+1)+1\big) \geq \Gamma\big(k(3-d+a)+1\big) \quad \mbox{for any} \ k \geq m_1,$$ where $m_1$ is chosen such that $m_1(3-d+a)>x_0$, and therefore for any $m\geq m_1$,
\begin{equation}
\label{sum-m-wave}
\sup_{\alpha \in [a,b]}\sum_{k\geq m+1}\bE|I_k^{\alpha}(f_{t,x,k})|^2 \leq \sum_{k\geq m+1}
\Gamma_{0,t}^k K_{d,a,b}^k \frac{2^k \, (t \vee 1)^{k(3-d+b)}}{\Gamma\big(k(3-d+a))+1\big)}.
\end{equation}
Relation \eqref{uniform} follows, since the series appearing in
\eqref{sum-m-heat} and \eqref{sum-m-wave} are clearly convergent.

\medskip

{\em Step 2. (tightness)}
The fact that the sequence $(u^{\alpha_n})_{n\geq 1}$ is tight in $C([0,T] \times \bR^d)$ follows by Proposition 2.3 of \cite{yor}, using Theorem \ref{Th-uniform-C} above.

\section{Rough Noise}
\label{section-rough}

In this section, we consider the case when the measure $\mu$ is given by:
\[
\mu(d\xi)=c_H|\xi|^{1-2H}d\xi \quad \mbox{for some $H \in (0,1/2)$}
\]
In addition, we assume that the temporal covariance function $\gamma_0$ is given by \eqref{def-gamma0}.

To emphasize the dependence on the parameter $H$, we denote the noise, the Hilbert space, and the solution, by $W^{H},\cH^{H},u^{H}$, respectively. The multiple integral of order $n$ with respect to $W^{H}$ will be denoted by $I_n^{H}$.
With this notation, the series expansion \eqref{series} becomes:
\[
u^H(t,x)=1+\sum_{k\geq 1}I_{n}^H(f_{t,x,k}).
\]

Similarly to Section \ref{section-regular},
we begin with a construction of all Gaussian processes $(W^{H})_{H \in (0,1/2)}$ on the same probability space.
Let $\widehat{W}$ be the $\bC$-valued Gaussian random measure given by \eqref{def-W-hat}. For any $\varphi \in \cS_{\bC}(\bR_{+}\times \bR)$, let
\[
W^{H}(\varphi)=\sqrt{c_{H_0} c_H}\int_{\bR^2}
\cF \varphi (\tau,\xi)\tau^{\frac{1}{2}-H_0}
|\xi|^{\frac{1}{2}-H}\widehat{W}(d\tau,d\xi),
\]
where $\cF \varphi (\tau,\xi)$ denotes the Fourier transform in both variables $(t,x)$. We will use the notation
$W^H(\varphi)=\int_0^{\infty}\int_{\bR}\varphi(t,x) W^H(dt,dx)$.

By the isometry property \eqref{iso-wide-W}, it follows that
\[
\bE[W^{H}(\varphi) W^{H}(\psi)]=c_{H_0}c_H\int_{\bR \times \bR}
\cF \varphi(\tau,\xi) \overline{\cF \psi(\tau,\xi) } |\tau|^{1-2H_0}|\xi|^{1-2H}d\tau d\xi=\langle \varphi,\psi \rangle_{\cH_{H}}.
\]
This shows that $W^H$ has the desired covariance function \eqref{cov}, and all processes $(W^H)_{H \in (0,1/2)}$ are defined on the same probability space. For any function $\varphi$ for which the stochastic integral is well-defined, we have:
\[
\int_0^{\infty} \int_{\bR} \varphi(t,x)W^H(dt,dx)=\sqrt{c_{H_0} c_H}\int_{\bR^2}
\cF_t [\cF_x \varphi(t,\cdot)(\xi)](\tau)|\tau|^{\frac{1}{2}-H_0}|\xi|^{\frac{1}{2}-H}
\widehat{W}(d\tau,d\xi).
\]
This representation can be extended to multiple integrals with respect to $W^{H}$:
\begin{equation}
\label{In-H}
I_k^{H}(\varphi)=c_{H_0}^{k/2} c_H^{k/2}\int_{(\bR \times \bR)^k} \cF_t[\cF_x \varphi(\pmb{t_k},\bullet)(\pmb{\xi_k})](\pmb{\tau_k})
\prod_{j=1}^{k}|\tau_j|^{\frac{1}{2}-H_0}|\xi_j|^{\frac{1}{2}-H}
\widehat{W}(d\tau_1,d\xi_1) \ldots \widehat{W}(d\tau_k,d\xi_k).
\end{equation}

We will use the following lemma, which can be found for instance in \cite{BJQ}.

\begin{lemma}
\label{rough-int}
For any $t>0$,
\[
\int_{\bR} e^{-t|\xi|^2}|\xi|^{\alpha}d\xi=\Gamma\left(
\frac{1+\alpha}{2}\right)t^{-\frac{1+\alpha}{2}} \quad \mbox{for any $\alpha>-1$}
\]

\[
\int_{\bR}\frac{\sin^2(t|\xi|)}{|\xi|^2}|\xi|^{\alpha}d\xi=
2^{1-\alpha}\widetilde{C}_{\alpha}t^{1-\alpha} \quad
\mbox{for any $\alpha\in (-1,1)$},
\]
where
\[
\widetilde{C}_{\alpha}=
\left\{
\begin{array}{ll}
(1-\alpha)^{-1} \Gamma(\alpha)\sin(\pi \alpha/2) & \mbox{if $\alpha \in (0,1)$} \\
\alpha^{-1}(1-\alpha)^{-1}\Gamma(1+\alpha)\sin(\pi \alpha/2) & \mbox{if $\alpha \in (-1,0)$}\\
\pi/2 & \mbox{if $\alpha=0$}
\end{array} \right.
\]
\end{lemma}

We will use a similar approximation technique as in Section \ref{section-regular}. More precisely, we study first the $L^2(\Omega)$-continuity in $H$ of the multiple integral of $I_k^H(f_{t,x,k})$ for fixed $k$.
This is achieved by the following lemma.

\begin{lemma}
\label{rough-conv-Ik}
Let $\ell$ be given by \eqref{def-ell}.
If $H_n \to H^* \in (\ell,1/2)$, then
\[
Q_n:=\bE|I_k^{H_n}(f_{t,x,k})-I_k^{H^*}(f_{t,x,k})|^2 \to 0, \quad \mbox{as $n\to \infty$},
\]
for any $t>0$, $x \in \bR$ and $k\geq 1$.
\end{lemma}

\begin{proof}
Similarly to \eqref{est-Qn}, we have:
\[
Q_n \leq k!\alpha_{H_0}^k \int_{T_k(t)} \int_{T_k(t)} \prod_{j=1}^k |t_j-s_j|^{2H_0-2}A_k^{(n)}(\pmb{t}_k,\pmb{s_k})
d\pmb{t}_k d\pmb{s_k},
\]
where
\[
A_k^{(n)}(\pmb{t}_k,\pmb{s_k})=\int_{\bR^k}
\phi_{\pmb{\xi_k}}(\pmb{t_k})
\phi_{\pmb{\xi_k}}(\pmb{s_k}) \left|c_{H_n}^{k/2}\prod_{j=1}^k|\xi_j|^{\frac{1}{2}-H_n}-
c_{H^*}^{k/2}\prod_{j=1}^k|\xi_j|^{\frac{1}{2}-H^*}\right|^2 d\pmb{\xi_k}.
\]
and $\phi_{\pmb{\xi_k}}(\pmb{t_k})$ is given by \eqref{def-phi-k}. We treat separately the heat and wave equation.

\medskip

{\bf Case 1. (wave equation)} In this case, we reduce our problem to the white noise in time. More precisely, using the inequality \eqref{A-ineq}, we have
\begin{align*}
Q_n & \leq k! \Gamma_{0,t}^k\int_{T_k(t)} A_k^{(n)}(\pmb{t_k},\pmb{t_k}) d \pmb{t_k} \\
& = k! \Gamma_{0,t}^k \int_{T_k(t)} \int_{\bR^k} \prod_{j=1}^k \frac{\sin^2((t_{j+1}-t_j)|\xi_1+\ldots+\xi_j|)}{|\xi_1+\ldots+\xi_j|^2} \left|c_{H_n}^{k/2}\prod_{j=1}^k|\xi_j|^{\frac{1}{2}-H_n}-
c_{H^*}^{k/2} \prod_{j=1}^k|\xi_j|^{\frac{1}{2}-H^*}\right|^2
  d\pmb{\xi_k} d \pmb{t_k}
\end{align*}
where $\Gamma_{0,t}=2H_0 t^{2H_0-1}$. The proof of Theorem 4.1 of \cite{GJQ} shows that the last integral converges to $0$ as $n\to \infty$.

\medskip

{\bf Case 2. (heat equation)} In this case, we still use the Cauchy-Schwarz inequality
$A_k^{(n)}(\pmb{t}_k,\pmb{s_k}) \leq A_k^{(n)}(\pmb{t}_k,\pmb{t_k})^{1/2} A_k^{(n)}(\pmb{s}_k,\pmb{s_k})^{1/2}$,
but we combine it with the following inequality, which is a consequence of the Littlewood-Hardy inequality: for any $\varphi \in L^{1/H_0}(\bR_{+}^k)$,
\begin{equation}
\label{LH-ineq}
\alpha_{H_0}^k \int_{\bR_{+}^k} \int_{\bR_{+}^k}
\prod_{j=1}^{k}|t_j-s_j|^{2H_0-2} \varphi(\pmb{t_k})
\varphi(\pmb{s_k}) d\pmb{t_k}d\pmb{s_k} \leq b_{H_0}^k\left(
\int_{\bR_{+}^k}|\varphi(\pmb{t_k})|^{1/H_0} d\pmb{t_k} \right)^{2H_0},
\end{equation}
where $b_{H_0}>0$ depends only on $H_0$. Hence,
$
Q_n \leq k! \, b_{H_0}^k \left(\int_{T_k(t)} A_k^{(n)}(\pmb{t}_k,\pmb{t_k})^{\frac{1}{2H_0}}d\pmb{t_k}\right)^{2H_0}.
$
Therefore, it suffices to show that:
\begin{equation}
\label{conv-A-H0}
\int_{T_k(t)} A_k^{(n)}(\pmb{t}_k,\pmb{t_k})^{\frac{1}{2H_0}}d\pmb{t_k}\to 0, \quad \mbox{as $n\to \infty$}.
\end{equation}

For this, we apply the Dominated Convergence Theorem. We will show that
\begin{equation}
\label{A-conv}
A_k^{(n)}(\pmb{t}_k,\pmb{t_k})\to 0 \quad \mbox{as $n\to \infty$},
\end{equation}
and there exists a function $h_{k}(\pmb{t_k})$ such that
\begin{equation}
\label{domin}
A_k^{(n)}(\pmb{t}_k,\pmb{t_k}) \leq h_{k}(\pmb{t_k}) \quad \mbox{and} \int_{T_k(t)}h_{k}(\pmb{t_k}) d\pmb{t_k}<\infty.
\end{equation}

We first prove \eqref{A-conv}. By the change of variables $\eta_j=\xi_1+\ldots+\xi_j$ for $j=1,\ldots,n$,
\begin{align*}
A_k^{(n)}(\pmb{t}_k,\pmb{t_k})&=\int_{\bR^k}\prod_{j=1}^{k}
e^{-(t_{j+1}-t_j)|\xi_1+\ldots+\xi_j|^2}
\left|c_{H_n}^{k/2}\prod_{j=1}^k|\xi_j|^{\frac{1}{2}-H_n}-
c_{H^*}^{k/2} \prod_{j=1}^k|\xi_j|^{\frac{1}{2}-H^*}\right|^2
  d\pmb{\xi_k}\\
&=\int_{\bR^k}\prod_{j=1}^{k}
e^{-(t_{j+1}-t_j)|\eta_j|^2}
\left|c_{H_n}^{k/2}\prod_{j=1}^k|\eta_j-\eta_{j-1}|^{\frac{1}{2}-H_n}-
c_{H^*}^{k/2} \prod_{j=1}^k|\eta_j-\eta_{j-1}|^{\frac{1}{2}-H^*}\right|^2
  d\pmb{\eta_k},
\end{align*}
where $\eta_0=0$. We denote by $B^{(n)}(\pmb{t_k},\pmb{\eta_k})$ the integrand in the last integral above. Clearly, $B^{(n)}(\pmb{t_k},\pmb{\eta_k}) \to 0$ as $n\to \infty$. So relation \eqref{A-conv} will follow by the Dominated Convergence Theorem, as long as we show that there exists a function $B(\pmb{t_k},\pmb{\eta_k})$ such that
\[
B^{(n)}(\pmb{t_k},\pmb{\eta_k})\leq B(\pmb{t_k},\pmb{\eta_k}) \quad \mbox{and} \quad \int_{\bR^k}B(\pmb{t_k},\pmb{\eta_k})d
\pmb{\eta_k}<\infty.
\]
Note that $B^{(n)}(\pmb{t_k},\pmb{\eta_k})\leq 2\big(B_1^{(n)}(\pmb{t_k},\pmb{\eta_k})+B_2^{(n)}(\pmb{t_k},
\pmb{\eta_k})\big)$, where
\[
B_1^{(n)}(\pmb{t_k},\pmb{\eta_k})=\prod_{j=1}^{k}
e^{-(t_{j+1}-t_j)|\eta_j|^2}
c_{H_n}^{k}\prod_{j=1}^k|\eta_j-\eta_{j-1}|^{1-2H_n},
\]
and $B_2^{(n)}$ has a similar expression with $H_n$ replaced by $H^*$. It suffices to consider $B_1^{(n)}$.

Fix numbers $a$ and $b$ such that
\begin{equation}
\label{def-ab}
\ell<a<H^*<b<\frac{1}{2}.
\end{equation}
Then $H_n \in [a,b]$ for $n$ large enough.
Since the constant $c_H$ defines a continuous function of $H$, $c_{H_n}$ is bounded by a constant $c$.
We use the inequality
\begin{equation}
\label{prod-ineq}
\prod_{j=1}^{k}|\eta_j-\eta_{j-1}|^{1-2H} \leq \sum_{\pmb{a} \in A_k} \prod_{j=1}^{k}|\eta_j|^{(1-2H)a_j},
\end{equation}
where $A_k$ is a set of multi-indices $\pmb{a}=(a_1,\ldots,a_k)$ such that
$a_1 \in \{1,2\}$, $a_n \in \{0,1\}$, $a_j \in \{0,1,2\}$ for $j=1,\ldots,k-1$, $\sum_{j=1}^{k}a_j=k$, and ${\rm card}(A_k)=2^{k-1}$. It follows that:
\[
B_1^{(n)}(\pmb{t_k},\pmb{\eta_k}) \leq c^{k} \prod_{j=1}^{k}
e^{-(t_{j+1}-t_j)|\eta_j|^2}
\sum_{\pmb{a} \in A_n}\prod_{j=1}^{k}|\eta_j|^{(1-2H_n)a_j}.
\]
Moreover,
$|\eta_j|^{(1-2H)a_j} \leq f_{a}(|\eta_j|)$ for any $j=1,\ldots,n$, where $a$ is the constant from \eqref{def-ab} and the functions $f_0,f_1,f_2$ are defined as follows: $f_0(r)=1$ for any $r> 0$,
\[
f_1(r)=
\left\{
\begin{array}{ll}
r^{1-2a}  & \mbox{if $r\geq 1$} \\
1 & \mbox{if $r\in (0,1)$}
\end{array} \right.
\quad
\mbox{and}
\quad
f_2(r)=
\left\{
\begin{array}{ll}
r^{2(1-2a)}  & \mbox{if $r\geq 1$} \\
1 & \mbox{if $r\in (0,1)$}
\end{array} \right.
\]
Hence,
\[
B_1^{(n)}(\pmb{t_k},\pmb{\eta_k}) \leq c^{k} \prod_{j=1}^{k}
e^{-(t_{j+1}-t_j)|\eta_j|^2}
\sum_{\pmb{a} \in A_n}\prod_{j=1}^{k}f_{a_j}(|\eta_j|)=:F(\pmb{t_k},\pmb{\eta_k}).
\]
By the same argument, $B_2^{(n)}(\pmb{t_k},\pmb{\eta_k})\leq F(\pmb{t_k},\pmb{\eta_k})$. Hence, $B^{(n)}(\pmb{t_k},\pmb{\eta_k})\leq 2F(\pmb{t_k},\pmb{\eta_k})$ and so,
\begin{equation}
\label{bound-A}
A_{k}^{(n)}(\pmb{t_k},\pmb{t_k})=\int_{\bR^k} B^{(n)}(\pmb{t_k},\pmb{\eta_k})d\pmb{\eta_k} \leq
2 \int_{\bR^k} F(\pmb{t_k},\pmb{\eta_k})d\pmb{\eta_k}.
\end{equation}
It remains to prove that $\int_{\bR^k} F(\pmb{t_k},\pmb{\eta_k})d\pmb{\eta_k}<\infty$. We give below an estimate for this integral which will be used for the proof of relation \eqref{domin} below. Note that
\[
\int_{\bR^k} F(\pmb{t_k},\pmb{\eta_k})d\pmb{\eta_k}=\sum_{\pmb{a} \in A_k} \prod_{j=1}^k I(a_j) \quad \mbox{with} \quad
I(a_j)=\int_{\bR}e^{-(t_{j+1}-t_j)|\eta_j|^2}f_{a_j}(|\eta_j|)d\eta_j.
\]
We treat separately the cases $a_j=0,1,2$. Using Lemma \ref{rough-int}, we have:
\begin{align*}
I(0)&=\int_{\bR}e^{-(t_{j+1}-t_j)|\eta_j|^2}d\eta_j=\sqrt{\pi}
(t_{j+1}-t_j)^{-1/2}\\
I(1)&=\int_{|\eta_j|\leq 1}e^{-(t_{j+1}-t_j)|\eta_j|^2}d\eta_j+\int_{|\eta_j|>1}
e^{-(t_{j+1}-t_j)|\eta_j|^2}|\eta_j|^{1-2a}d\eta_j\\
&\leq \sqrt{\pi}(t_{j+1}-t_j)^{-1/2}+\Gamma(1-a)(t_{j+1}-t_j)^{-(1-a)}\\
I(2)&=\int_{|\eta_j|\leq 1}e^{-(t_{j+1}-t_j)|\eta_j|^2}d\eta_j+\int_{|\eta_j|>1}
e^{-(t_{j+1}-t_j)|\eta_j|^2}|\eta_j|^{2(1-2a)}d\eta_j\\
&\leq \sqrt{\pi}(t_{j+1}-t_j)^{-1/2}+\Gamma\left(\frac{3-4a}{2}\right)
(t_{j+1}-t_j)^{-\frac{3-4a}{2}}.
\end{align*}
Therefore, there exists a constant $c_a>0$ depending on $a$ such that
\begin{align*}
I(a_j) & \leq c_a\left\{ (t_{j+1}-t_j)^{-1/2}+
(t_{j+1}-t_j)^{-(1-a)}+(t_{j+1}-t_j)^{-\frac{3-4a}{2}}\right\}\\
&\leq
c_a(t^{1-2a}+t^{\frac{1}{2}-a}+1)(t_{j+1}-t_j)^{-\frac{3-4a}{2}}=:
c_{a,t}(t_{j+1}-t_j)^{-\frac{3-4a}{2}}.
\end{align*}
Since ${\rm card}(A_k)=2^{k-1}$, we obtain:
\begin{equation}
\label{est-F}
\int_{\bR^k} F(\pmb{t_k},\pmb{\eta_k})d\pmb{\eta_k}\leq
2^{k-1} c_{a,t}(t_{j+1}-t_j)^{-\frac{3-4a}{2}}.
\end{equation}
This proves the integrability of $F(\pmb{t_k},\cdot)$ and relation \eqref{A-conv} follows by the Dominated Convergence Theorem.

\medskip
Next, we prove \eqref{domin}. By \eqref{bound-A} and \eqref{est-F}, we have:
\begin{align*}
A_{k}^{(n)}(\pmb{t_k},\pmb{t_k})^{\frac{1}{2H_0}}\leq c_{a,t}^{\frac{k}{2H_0}}\prod_{j=1}^{k}
(t_{j+1}-t_j)^{-\frac{3-4a}{4H_0}}=:h_k(\pmb{t_k}).
\end{align*}
The function $h_k$ is integrable on $T_k(t)$ since $\frac{3-4a}{4H_0}<1$, due to our choice of $a$ in \eqref{def-ab}. This finishes the proof of \eqref{domin} and the justification of the application of the Dominated Convergence Theorem.
\end{proof}

We now prove some moment estimates for the increments of the solution $u^H$, uniform in $H$, which will be used for the proof of tightness of the family $(u^H)_{H \in [a,b]}$ below. In particular, Theorem \ref{rough-unif-mom} shows that the process $u^H$ has a continuous modification. We will work with this modification.

\begin{theorem}
\label{rough-unif-mom}
Let $\ell$ be given by \eqref{def-ell}. Let $[a,b]$ be a compact set in $[0,1]$ such that:
\begin{equation}
\label{range-ab}
\ell<a<b<\frac{1}{2}.
\end{equation}

(a) Let $u^H$ be the solution of the heat equation \eqref{pam} with noise $W^H$. For any $p\geq 2$, $T>0$, $c_0 \in (0,\frac{2H_0+a-1}{2H_0})$ and
\begin{equation}
\label{cond-delta}
0<\delta<2H_0(1-c_0)+a-1,
\end{equation}
there exist a constant $C>0$ such that for any $t,t'\in [0,T]$ and $x,x'\in \bR$,
\[
\sup_{H \in [a,b]}\bE|u(t,x)-u(t',x')|^p \leq C \Big(|t-t'|^{p\delta/2}+
|x-x'|^{p\delta}\Big).
\]

(b) Let $u^H$ be the solution of the wave equation \eqref{ham} with noise $W^H$. For any $p\geq 2$, $T>0$,
there exist a constant $C>0$ such that for any $\delta \in (0,a)$ and $t,t'\in [0,T]$,
\[
\sup_{H \in [a,b]}\bE|u(t,x)-u(t',x)|^p \leq C |t-t'|^{p\delta}.
\]
For any $p\geq 2$, $T>0$ and $c_1 \in (0,2a)$,
there exist a constant $C>0$ such that for any $\delta \in (\frac{c_1}{2},a)$ and $x,x'\in \bR$,
\[
\sup_{H \in [a,b]}\bE|u(t,x)-u(t,x')|^p \leq C
|x-x'|^{p\delta}.
\]
\end{theorem}

\begin{proof}
We examine separately the time increments and the space increments.
We denote by $c$ a constant that may depend on $p,T,H_0,a,b$ and $c_0$ and may be different in each of its appearances.

{\bf Step 1 (time increments).} Let $t\in [0,T]$ and $h>0$ be such that $t+h \in [0,T]$.
As in the proof of Theorem \ref{Th-uniform-C},
\begin{align*}
\big\|u^{H}(t+h,x)-u^{H}(t,x)\big\|_p & \leq
\sum_{n\geq 1}(p-1)^{n/2}\left(\frac{2}{n!} \Big(A_n^{H}(t,h)+B_n^{H}(t,h)\Big) \right)^{1/2},
\end{align*}
where
$A_n^{H}(t,h)=(n!)^2\| \widetilde{f}_{t+h,x,n}1_{[0,t]^n}-\widetilde{f}_{t,x,n}
\|_{\cH_{H}^{\otimes n}}^2$ and $B_n^{H}(t,h)=(n!)^2\| \widetilde{f}_{t+h,x,n}1_{[0,t+h]^n -[0,t]^n}\|_{\cH_{H}^{\otimes n}}^2$.

We study separately the two terms, for the heat equation and for the wave equation.

\medskip

{\bf Study of $A_n^{H}(t,h)$ (heat equation).} By the Littlewood-Hardy inequality \eqref{LH-ineq},
\begin{equation}
\label{LH-A}
A_n^{H}(t,h) \leq  b_{H_0}^n \left(\int_{[0,t]^n} \psi_{t,h,n}^{H}(\pmb{t_n})^{\frac{1}{2H_0}}d \pmb{t_n}\right)^{2H_0},
\end{equation}
where the function $\psi_{t,h,n}^{H}(\pmb{t_n})$ is defined as follows: if $\pmb{t_n} \in [0,t]^n$ and $\rho$ is the permutation of $1,\ldots,n$ such that $t_{\rho(1)}<\ldots<t_{\rho(n)}$, we let $u_j=t_{\rho(j+1)}-t_{\rho(j)}$ (with $t_{\rho(n+1)}=t$), and
\begin{align*}
\psi_{t,h,n}^{H}(\pmb{t_n}) & : =
c_{H}^n \int_{\bR^n} \prod_{j=1}^{n-1}\big|\cF G_{u_j}(\xi_1+\ldots+\xi_j)\big|^2 \big|\cF (G_{u_n+h}-G_{u_n})(\sum_{k=1}^n\xi_k)\big|^2 \prod_{j=1}^{n}
|\xi_j|^{1-2H}d\pmb{\xi_n}.
\end{align*}
Using the change of variables $\eta_j=\xi_1+\ldots+\xi_j$ for $j=1,\ldots,n$ (with $\eta_0=0$) followed by the inequality \eqref{prod-ineq}, we see that
\begin{align}
\nonumber
& \psi_{t,h,n}^{H}(\pmb{t_n})  = c_H^n
\int_{\bR^n} \prod_{j=1}^{n-1}\big|\cF G_{u_j}(\eta_j)\big|^2 \big|\cF (G_{u_n+h}-G_{u_n})(\eta_n)\big|^2 \prod_{j=1}^{n}
|\eta_j-\eta_{j-1}|^{1-2H}d\pmb{\eta_n}\\
\label{psi-bound}
& \quad \leq c_{H}^n \sum_{\pmb{\alpha_n}\in D_n^{(H)}} \prod_{j=1}^{n-1}\left(\int_{\bR}|\cF G_{u_j}(\eta_j)|^2 |\eta_j|^{\alpha_j} d\eta_j\right) \left( \int_{\bR}
\big|\cF (G_{u_n+h}-G_{u_n})(\eta_n)\big|^2 |\eta_n|^{\alpha_n} d\eta_n
\right)
\end{align}
where $D_n^{(H)}$ is the set of multi-indices $\pmb{\alpha_n}=(\alpha_1,\ldots,\alpha_n)$ with $\alpha_j=(1-2H)a_j$ and $\pmb{a_n}=(a_1,\ldots,a_n)\in A_n$. We use Lemma \ref{rough-int} to compute the $d\eta_j$ integrals with $j=1,\ldots,n-1$. (We use the notation $G$ here since we use this relation for both heat and wave equations.)

For the $d\eta_n$ integral, we use the inequality
$1-e^{-x}\leq x^{\e}$ for any $x>0$ and $\e \in [0,1]$, and obtain
\begin{align*}
& \int_{\bR}
\big|\cF (G_{u_n+h}^h-G_{u_n}^h)(\eta_n)\big|^2 |\eta_n|^{\alpha_n}d\eta_n=
\int_{\bR}e^{-u_n|\eta_n|^2}\left(1-e^{-h|\eta_n|^2/2}\right)^2
 |\eta_n|^{\alpha_n} d\eta_n\\
& \quad \quad \quad \leq h^{2\e}\int_{\bR}e^{-u_n|\eta_n|^2} |\eta_n|^{\alpha_n+4\e}d\eta_n=
h^{2\e} \Gamma\left(\frac{1+\alpha_n+4\e}{2}\right)
 u_n^{-\frac{1+\alpha_n+4\e}{2}}.
\end{align*}
(We used the notation $G^h$ here to emphasize that this calculation is valid only for the heat equation.) Therefore,
\begin{align*}
\psi_{t,h,n}^{H}(\pmb{t_n})  &\leq h^{2\e} c_H^n \sum_{\pmb{\alpha_n}\in D_n^{(H)}}
\prod_{j=1}^{n-1}\Gamma\left( \frac{1+\alpha_j}{2}\right)
u_j^{-\frac{1+\alpha_j}{2}}
\Gamma\left(\frac{1+\alpha_n+4\e}{2}\right)
u_n^{-\frac{1+\alpha_n+4\e}{2}} \\
& \leq h^{2\e} c_H^n C_{H,1}^{n-1} \sum_{\pmb{\alpha_n}\in D_n^{(H)}}\prod_{j=1}^{n-1}u_j^{-\frac{1+\alpha_j}{2}} u_n^{-\frac{1+\alpha_n+4\e}{2}}
\end{align*}
using the fact that $\Gamma\left(\frac{1+\alpha_n+4\e}{2}\right) \leq \Gamma(3)=2$ and
$\prod_{j=1}^{n-1}\Gamma\left( \frac{1+\alpha_j}{2}\right)\leq C_{H,1}^{n-1}$, where
\begin{equation}
\label{def-CH1}
C_{H,1}=\max\left\{\Gamma\left(\frac{1}{2}\right),\Gamma(1-H),\Gamma\left(\frac{3-4H}{2} \right) \right\}.
\end{equation}
Both $c_H$ and $C_{H,1}$ can be uniformly bounded for all $H \in [a,b]$. From here, we derive that
\begin{align}
\label{psi-b}
\int_{[0,t]^n}\left(\psi_{t,h,n}^{H}(\pmb{t_n})\right)^{\frac{1}{2H_0}} d\pmb{t_n} \leq c^{n-1}h^{\frac{\e}{H_0}} n! \sum_{\pmb{\alpha_n} \in D_n^{(H)}}\int_{T_n(t)}\prod_{j=1}^{n-1}
(t_{j+1}-t_j)^{-\frac{1+\alpha_j}{4H_0}}
(t-t_n)^{-\frac{1+\alpha_n+2\delta}{4H_0}}d\pmb{t_n}.
\end{align}
To compute the integral
\[
{\cal I}(\pmb{\alpha_n}):=\int_{T_n(t)}\prod_{j=1}^{n-1}
(t_{j+1}-t_j)^{-\frac{1+\alpha_j}{4H_0}}
(t-t_n)^{-\frac{1+\alpha_n+2\delta}{4H_0}}d\pmb{t_n}
\]
we use Lemma \ref{beta-lem} with $\beta_j=-\frac{1+\alpha_j}{4H_0}$ for $j=1,\ldots,n-1$ and $\beta_n=-\frac{1+\alpha_n+2\e}{4H_0}$. To apply this lemma, we need to check that $\beta_j>-1$ for all $j=1,\ldots,n$. When $\alpha_j =2(1-2H)$, we use the condition $4H_0+4H>3$. When $\alpha_n=1-2H$ and $H \in [a,b]$ is arbitrary, we encounter the condition
\begin{equation}
\label{cond-eps}
\e<\frac{2H_0+a-1}{2}.
\end{equation}
Under this condition, we can apply Lemma \ref{beta-lem} to deduce that:
\begin{align*}
{\cal I}(\pmb{\alpha_n}) = \frac{\prod_{j=1}^{n-1}\Gamma(1-\frac{1+\alpha_j}{4H_0})
\Gamma(1-\frac{1+\alpha_n+4\e}{4H_0})}
{\Gamma(\frac{2H_0+H-1}{2H_0}n-\frac{\e}{H_0}+1)}
t^{\frac{2H_0+H-1}{2H_0}n-\frac{\e}{H_0}}.
\end{align*}

To bound these factors, we use some monotonicity properties of the Gamma function: $\Gamma$ is decreasing on $(0,x_0)$ and increasing on $(x_0,\infty)$, where $x_0 \approx 1.4$. Therefore, for any $\alpha_j \in \{0,1-2H,2(1-2H)\}$ and $H \in [a,b]$,
we have $\prod_{j=1}^{n-1}\Gamma(1-\frac{1+\alpha_j}{4H_0}) \leq \Gamma(1-\frac{3-4a}{4H_0})^{n-1}$.
Next, we observe that $1-\frac{1+\alpha_n+4\e}{4H_0}$ takes values in the interval $[1-\frac{1-a+2\e}{2H_0},1-\frac{1}{4H_0}]$ whose lower bound may be close to 0. Since $\lim_{x\to 0}\Gamma(x)=\infty$, we control this term by fixing an arbitrary value $c_0\in (0,\frac{2H_0+a-1}{2H_0})$, and then
choosing
\begin{equation}
\label{choice-of-e}
\e=\delta/2 \quad \mbox{for some} \quad 0<\delta<2H_0(1-c_0)+a-1.
\end{equation}
 With this choice of $\e$, \eqref{cond-eps} holds, and
more importantly $1-\frac{1-a+2\e}{2H_0}>c_0$, so that for any $\alpha_n\in \{0,1-2H\}$ and $H \in [a,b]$,
\begin{equation}
\label{G-c0}
\Gamma\left(1-\frac{1+\alpha_n+4\e}{4H_0}\right)\leq \Gamma(c_0).
\end{equation}
To bound the $\Gamma$-value appearing in the denominator, we pick an integer $m_0$ such that $(m_0-1)\frac{2H_0+a-1}{2H_0}>x_0$. Then, for any $n\geq m_0$ and for any $H \in [a,b]$,
\[
\Gamma\left(\frac{2H_0+H-1}{2H_0}n-\frac{\e}{H_0}+1\right)>
\Gamma\left(\frac{(n-1)(2H_0+a-1)}{2H_0}+1\right)>c^{n-1}
[(n-1)!]^{\frac{2H_0+a-1}{2H_0}}.
\] Hence, for any $\pmb{\alpha_n} \in D_{n}^{(H)}$ and for any $H \in [a,b]$,
\begin{align*}
{\cal I}(\pmb{\alpha_n}) \leq \frac{c^{n-1}}{[(n-1)!]^{\frac{2H_0+a-1}{2H_0}}}(t\vee 1)^{\frac{n(2H_0+b-1)}{2H_0}}.
\end{align*}
Returning to \eqref{psi-b}, we have:
\[
\int_{[0,t]^n}\Big(\psi_{t,h,n}^{H}(\pmb{t_n})\Big)^{\frac{1}{2H_0}}
d\pmb{t_n} \leq h^{\frac{\e}{H_0}} \frac{c^{n-1}}{[(n-1)!]^{\frac{a-1}{2H_0}}}(t\vee 1)^{\frac{n(2H_0+b-1)}{2H_0}}
\]

Finally, coming back to \eqref{LH-A}, we obtain:
\[
A_n^H(t,h) \leq h^{2\e}   \frac{c^{n-1}}{[(n-1)!]^{a-1}}(t\vee 1)^{n(2H_0+b-1)}.
\]
Consequently, for any $p\geq 2$, $t \in [0,T]$, $H \in [a,b]$ and $\e$ as in \eqref{choice-of-e},
\begin{equation}
\label{bound-A-H}
\sum_{n\geq 1}(p-1)^{n/2}\left(\frac{1}{n!}A_n^H(t,h)\right)^{1/2} \leq C h^{\e}.
\end{equation}

\medskip

\medskip

{\bf Study of $A_n^{H}(t,h)$ (wave equation).} By the Cauchy-Schwarz inequality,
\begin{equation}
\label{CS-A}
A_n^{H}(t,h) \leq  \Gamma_{0,t}^{n} \int_{[0,t]^n} \psi_{t,h,n}^{H}(\pmb{t_n})d \pmb{t_n}
\end{equation}
with $\Gamma_{0,t}=2H_0 t^{2H_0-1}$ and the same function $\psi_{t,h,n}^{H}(\pmb{t_n})$ as above.

To estimate $\psi_{t,h,n}^{H}(\pmb{t_n})$, we use again \eqref{psi-bound}. The $d\eta_j$ integrals for $j=1,\ldots,n-1$ are computed using Lemma \ref{rough-int}. For the $d\eta_n$ integral, we consider separately the cases $\alpha_n=0$ and $\alpha_n=1-2H$.
If $\alpha_n=0$, by Plancherel theorem,
\begin{equation}
\label{w-b1}
\int_{\bR}
\big|\cF (G_{u_n+h}^w-G_{u_n}^w)(\eta_n)\big|^2 d\eta_n=
2\pi\int_{\bR}
\left( 1_{\{|x|<u_n+h\}}-1_{\{|x|<u_n\}}\right)^2 d\eta_n=4\pi h,
\end{equation}
and if $\alpha_n=1-2H$,
\begin{align*}
& \int_{\bR}
\big|\cF (G_{u_n+h}^w-G_{u_n}^w)(\eta_n)\big|^2 |\eta_n|^{1-2H}d\eta_n=\int_{\bR}
\frac{|\sin((u_n+h)|\eta_n|)-\sin(u_n|\eta_n|)|^2}{|\eta_n|^2}
|\eta_n|^{1-2H}d\eta_n.
\end{align*}
(We used the notation $G^w$ here to emphasize that this calculation is valid only for the wave equation.)
We write the last integral as $I_1+I_2$, where $I_1$ and $I_2$ are the integrals over the the regions
$|\eta_n|\leq 1$, respectively $|\eta_n|>1$. For $I_1$, we use that fact that
\[
|\sin((u_n+h)|\eta_n|)-\sin(u_n|\eta_n|)|^2=4 \sin^2\left(\frac{h|\eta_n|}{2}\right)\cos^2 \left(
\frac{(2u_n+h)|\eta_n|}{2}\right)\leq h^2 |\eta_n|^2
\]
and hence $I_1 \leq h^2 \int_{|\eta_n|\leq 1}|\eta_n|^{1-2H}d\eta_n =\frac{1}{1-H}h^2$. For $I_2$, writing $\sin(x)=\frac{e^{ix}-e^{-ix}}{2i}$ we have
\begin{align*}
|\sin((u_n+h)|\eta_n|)-\sin(u_n|\eta_n|)|^2 & =|e^{iu_n |\eta_n|}(e^{ih|\eta_n|}-1)-e^{-iu_n |\eta_n|}(e^{-ih|\eta_n|}-1)|^2\\
& \leq 2 \left(|(e^{ih|\eta_n|}-1|^2+|e^{-ih|\eta_n|}-1|^2 \right) \leq 4 h^{2\delta} |\eta_n|^{2\delta},
\end{align*}
for any $\delta \in [0,1]$, using the fact that $|1-e^{ix}|^2 \leq |x|^{2\delta}$ for any $x>0$ and $\delta \in [0,1]$.
Hence, $I_2 \leq h^{2} \int_{|\eta_n|>1}|\eta_n|^{2\delta-2H-1}d\eta_n=\frac{1}{2H-2\delta}h^{2\delta}$,
for any $\delta \in (0,H)$. To ensure that this condition holds for all $H \in [a,b]$, we choose $\delta \in(0,a)$. We obtain that for any $H \in [a,b]$
\begin{equation}
\label{w-b2}
\int_{\bR}
\big|\cF (G_{u_n+h}^w-G_{u_n}^w)(\eta_n)\big|^2 |\eta_n|^{1-2H}d\eta_n \leq \left(\frac{1}{1-b}+\frac{1}{2a-2\delta}\right)h^{2\delta}.
\end{equation}
From \eqref{w-b1} and \eqref{w-b2}, we obtain that for any $\alpha_n \in \{0,1-2H\}$, $H\in [a,b]$ and $\delta \in (0,a)$,
\begin{align*}
& \int_{\bR}
\big|\cF (G_{u_n+h}^w-G_{u_n}^w)(\eta_n)\big|^2 |\eta_n|^{\alpha_n}d\eta_n \leq c h^{2\delta}.
\end{align*}
Returning to \eqref{psi-bound}, we obtain:
\[
\psi_{t,h,n}^{H}(\pmb{t_n})  \leq c h^{2\delta} c_H^n \sum_{\pmb{\alpha_n}\in D_n^{(H)}}
\prod_{j=1}^{n-1}2^{1-\alpha_j}\widetilde{C}_{\alpha_j}
u_j^{1-\alpha_j}.
\]
Note that $\prod_{j=1}^{k-1}2^{1-\alpha_j}\widetilde{C}_{\alpha_j} \leq C_{H,1}^{k-1}$, where
\begin{equation}
\label{def-CH1-w}
C_{H,1}=\max\left\{\pi, \frac{\Gamma(1-2H)}{H}, \frac{2\Gamma(2-4H)}{4H-1}\right\}.
\end{equation}
Clearly, $c_H$ and $C_{H,1}$ can be bounded for all $H \in [a,b]$. Hence, recalling \eqref{CS-A}, we have:
\begin{align*}
A_n^H(t,h) \leq \int_{[0,t]^n}\psi_{t,h,n}^{H}(\pmb{t_n})d\pmb{t_n} & \leq h^{2\delta}c^{n-1} n!\sum_{\pmb{\alpha_n}\in D_n^{(H)}}\int_{T_n(t)}\prod_{j=1}^{n}
(t_{j+1}-t_j)^{1-\alpha_j}d\pmb{t_n} \\
&= h^{2\delta}c^{n-1} n!\sum_{\pmb{\alpha_n}\in D_n^{(H)}}
\frac{\prod_{j=1}^{n-1}\Gamma(2-\alpha_j)
}{\Gamma(n(2H+1)+\alpha_n)}\,t^{n(2H+1)+\alpha_n-1} \\
& \leq h^{2\delta} c^{n-1} \frac{(t\vee 1)^{n(2b+1)-2a}}{(n!)^{2a}},
\end{align*}
using the fact that $\prod_{j=1}^{n-1}\Gamma(2-\alpha_j) \leq 1$ and
$\Gamma(n(2H+1)+\alpha_n) \geq \Gamma(n(2a+1))\geq c^n (n!)^{2a+1}$.

Consequently, for any $p\geq 2$, $t \in [0,T]$, $H \in [a,b]$ and $\delta \in (0,a)$,
\[
\sum_{n\geq 1}(p-1)^{n/2}\left(\frac{1}{n!}A_n^H(t,h)\right)^{1/2} \leq C h^{\delta}.
\]

{\bf Study of $B_n^H(t,h)$ (heat equation).} By the Littlewood-Hardy inequality \eqref{LH-ineq},
\begin{equation}
\label{B-bd}
B_n^H(t,h) \leq b_{H_0}^n \left(\int_{[0,t+h]^n} \gamma_{t,h,n}^{(H)}(\pmb{t_n})^{\frac{1}{2H_0}}1_{D_{t,h}}
(\pmb{t_n})d\pmb{t_n},
\right)^{2H_0}
\end{equation}
where $D_{t,h}=[0,t+h]^n \verb2\2 [0,t]^n$, and the function $\gamma_{t,h,n}^{(H)}(\pmb{t_n})$ is defined as follows: if $\pmb{t}_n \in [0,t]^n$ and $\rho$ is the permutation of $1,\ldots,n$ such that $t_{\rho(1)}<\ldots<t_{\rho(n)}$ and $t<t_{\rho(n)}<t+h$, we let
$u_j=t_{\rho(j+1)}-t_{\rho(j)}$ for all $j=1,\ldots,n-1$  $u_n=t-t_{\rho(n)}$, and
\begin{align*}
\gamma_{t,h,n}^{(H)}(\pmb{t_n})&:=c_H^n \int_{\bR^n}\prod_{j=1}^{n-1}|\cF G_{u_j}(\xi_1+\ldots+\xi_j)|^2
|\cF G_{u_n+h}(\xi_1+\ldots+\xi_n)|^2 \prod_{j=1}^{n}|\xi_j|^{1-2H}d\pmb{\xi_n}.
\end{align*}
Using the change of variables $\eta_j=\xi_1+\ldots+\xi_j$ for $j=1,\ldots,n$ (with $\eta_0=0$) followed by inequality \eqref{prod-ineq}, we obtain:
\begin{align}
\label{b-gamma}
\gamma_{t,h,n}^{(H)}(\pmb{t_n}) & \leq c_H^n \sum_{\pmb{\alpha_n}\in D_n^{(H)}} \prod_{j=1}^{n-1}\left( \int_{\bR}|\cF G_{u_j}(\eta_j)|^2 |\eta_j|^{\alpha_j}d\eta_j\right)\left( \int_{\bR}|\cF G_{u_n+h}(\eta_j)|^2 |\eta_n|^{\alpha_n}d\eta_n \right).
\end{align}
(We will the notation $G$ here since we use this relation for both heat and wave equation.)

Using Lemma \ref{rough-int}, we obtain:
\[
\gamma_{t,h,n}^{(H)}(\pmb{t_n}) \leq c_H^n \sum_{\pmb{\alpha_n}\in D_n^{(H)}} \prod_{j=1}^{n}\Gamma\left(\frac{1+\alpha_j}{2}\right)u_j^{-\frac{1+\alpha_j}{2}}
(u_n+h)^{-\frac{1+\alpha_n}{2}}.
\]
As noted above, $\prod_{j=1}^{n}\Gamma\left(\frac{1+\alpha_j}{2}\right) \leq C_{H,1}^n$ where $C_{H,1}$ is given by \eqref{def-CH1} and is bounded for all $H \in [a,b]$. Taking power $\frac{1}{2H_0}$, we obtain:
\[
\gamma_{t,h,n}^{(H)}(\pmb{t_n}) \leq c^n \sum_{\pmb{\alpha_n}\in D_n^{(H)}} \prod_{j=1}^{n}
u_j^{-\frac{1+\alpha_j}{4H_0}}
(u_n+h)^{-\frac{1+\alpha_n}{4H_0}}.
\]
Integrating over the set $D_{t,h}$, we obtain:
\[
\int_{[0,t+h]^n}\gamma_{t,h,n}^{(H)}(\pmb{t_n})^{\frac{1}{2H_0}}
1_{D_{t,h}}(\pmb{t_n}) d\pmb{t_n} \leq c^n n!
\sum_{\pmb{\alpha_n}\in D_n^{(H)}} \int_t^{t+h}J^h(t_n)(t+h-t_n)^{-\frac{1+\alpha_n}{2}}dt_n,
\]
where
\[
J^h(t_n):=\int_{T_{n-1}(t_n)}\prod_{j=1}^{n-1}
(t_{j+1}-t_j)^{-\frac{1+\alpha_n}{4H_0}}d\pmb{t_{n-1}}\leq \frac{c^{n-1}}{[(n-1)!]^{\frac{4H_0+2a-2}{4H_0}}}.
\]
We conclude that
\[
\int_{[0,t+h]^n}\gamma_{t,h,n}^{(H)}(\pmb{t_n})^{\frac{1}{2H_0}}
1_{D_{t,h}}(\pmb{t_n}) d\pmb{t_n} \leq  \frac{c^{n-1}}{[(n-1)!]^{\frac{2a-2}{4H_0}}}\sum_{\pmb{\alpha_n}\in D_n^{(H)}}
\frac{1}{1-\frac{1+\alpha_n}{4H_0}}h^{1-\frac{1+\alpha_n}{4H_0}}.
\]
Note that $(1-\frac{1+\alpha_n}{4H_0})^{-1}$ is bounded by $\frac{2H_0}{2H_0+a-1}$ for any $\alpha_n \in \{0,1-2H\}$ and $H \in [a,b]$. Moreover, due to condition \eqref{cond-eps},
\[
\frac{\e}{H_0}<\frac{2H_0+a-1}{2H_0}\leq 1-\frac{1+\alpha_n}{4H_0}\leq 1-\frac{1}{4H_0}.
\]
Therefore,
\[
\int_{[0,t+h]^n}\gamma_{t,h,n}^{(H)}(\pmb{t_n})^{\frac{1}{2H_0}}
1_{D_{t,h}}(\pmb{t_n}) d\pmb{t_n} \leq h^{\frac{\e}{H_0}} \frac{c^{n-1}}{[(n-1)!]^{\frac{2a-2}{4H_0}}}.
\]
Returning to \eqref{B-bd}, we infer that
$B_n(t,h) \leq  h^{2\e} \frac{c^{n-1}}{[(n-1)!]^{a-1}}$,
and hence, for any $p\geq 2$, $t \in [0,T]$ and $\e$ satisfying \eqref{cond-eps} (and in particular, for any $\e$ chosen as in \eqref{choice-of-e}),
\[
\sum_{n\geq 1}(p-1)^{n/2}\left(\frac{1}{n!}B_n(t,h)  \right)^{1/2} \leq C h^{\e}.
\]

\medskip

{\bf Study of $B_n^H(t,h)$ (wave equation).} By the Cauchy-Schwarz equation,
\begin{equation}
\label{B-nth-wave}
B_{n}^H(t,h) \leq \Gamma_{0,t}^n \int_{[0,t+h]^n}\gamma_{t,h,n}^{(H)}(\pmb{t_n})
1_{D_{t,h}}(\pmb{t_n}) d\pmb{t_n},
\end{equation}
where the function $\gamma_{t,h,n}^{(H)}(\pmb{t_n})$
is defined as above. As mentioned above, to estimate $\gamma_{t,h,n}^{(H)}(\pmb{t_n})$ we use  \eqref{b-gamma}, which holds also for the wave equation. Using Lemma \ref{rough-int}, we obtain:
\[
\gamma_{t,h,n}^{(H)}(\pmb{t_n}) \leq c_H^n \sum_{\pmb{\alpha_n}
\in D_n^{(H)}} \prod_{j=1}^{n}2^{1-\alpha_j}\widetilde{C}_{\alpha_j}u^{1-\alpha_j} \prod_{j=1}^{n-1}u_j^{1-\alpha_j}(u_n+h)^{1-\alpha_n}.
\]
As noticed above, $\prod_{j=1}^{n}2^{1-\alpha_j}\widetilde{C}_{\alpha_j} \leq C_{H,1}^{n}$, where $C_{H,1}$ is given by \eqref{def-CH1}, and the constants $c_H$ and $C_{H,1}$ are uniformly bounded in $H \in [a,b]$. Hence,
\[
\int_{[0,t+h]^n}\gamma_{t,h,n}^{(H)}(\pmb{t_n})
1_{D_{t,h}}(\pmb{t_n}) d\pmb{t_n} \leq c^{n}n!
\sum_{\pmb{\alpha_n}
\in D_n^{(H)}} \int_t^{t+h} J^w(t_n)(t+h-t_n)^{1-\alpha_n}dt_n,
\]
where
\begin{equation}
\label{def-Jw}
J^w(t_n):=\int_{T_{n-1}(t_n)}\prod_{j=1}^{n-1}
(t_{j+1}-t_{j})^{1-\alpha_j}d\pmb{t_{n-1}}\leq \frac{c^{n-1}}{[(n-1)!]^{2a+1}}.
\end{equation}
We conclude that
\[
\int_{[0,t+h]^n}\gamma_{t,h,n}^{(H)}(\pmb{t_n})
1_{D_{t,h}}(\pmb{t_n}) d\pmb{t_n} \leq \frac{c^{n-1}}{[(n-1)!]^{2a}} \sum_{\pmb{\alpha_n}
\in D_n^{(H)}} \frac{1}{2-\alpha_n}h^{2-\alpha_n}.
\]
Note that $(2-\alpha_n)^{-1}$ is bounded by $2$ for any $\alpha_n \in \{0,1-2H\}$ and $H \in [a,b]$. Moreover, for any $h \in (0,T)$ and $\delta \in (0,a)$, $h^{2-\alpha_n} \leq c h^{2a+1}\leq ch^{2\delta+1}\leq c h^{2\delta}$.
Therefore,
\[
B_n^H(t,h)\leq \int_{[0,t+h]^n}\gamma_{t,h,n}^{(H)}(\pmb{t_n})
1_{D_{t,h}}(\pmb{t_n}) d\pmb{t_n} \leq h^{2\delta}\frac{c^{n-1}}{[(n-1)!]^{2a}}.
\]
Consequently, for any $p\geq 2$, $t \in[0,T]$, $H \in [a,b]$ and $\delta \in (0,a)$,
\[
\sum_{n\geq 1}(p-1)^{n/2}\left(\frac{1}{n!}B_n^H(t,h) \right)^{1/2} \leq c h^{\delta}.
\]

\medskip
{\bf Step 2. (space increments)} Note that for any $x,z \in \bR$,
\[
\|u(t,x+z)-u(t,x)\|_p \leq \sum_{n\geq 1}(p-1)^{n/2}\left( \frac{1}{n!}C_{n}^H(t,z)\right)^{1/2},
\]
where $C_{n}^H(t,z)=(n!)^2 \|\widetilde{f}_n(\cdot,t,x+z)-
\widetilde{f}_n(\cdot,t,x)\|_{\cH_H^{\otimes n}}^{2}$.
We study separately $C_{n}^H(t,z)$, for the heat equation and for the wave equation.

\medskip

{\bf Study of $C_n^{H}(t,z)$ (heat equation).}
By the Littlewood-Hardy inequality \eqref{LH-ineq},
\begin{equation}
C_n^{H}(t,z) \leq b_{H_0}^2 \left( \int_{[0,t]^n} \psi_{t,z,n}^{H}(\pmb{t_n})^{\frac{1}{2H_0}} d\pmb{t_n})\right)^{2H_0},
\end{equation}
where the function $\psi_{t,z,n}^{H}(\pmb{t_n})$ is defined as follows: if $\pmb{t_n}=(t_1,\ldots,t_n) \in [0,t]^n$ and $\rho$ is the permutation such that $t_{\rho(1)}<\ldots<t_{\rho(n)}$, we let $u_j=t_{\rho(j+1)}-t_{\rho(j)}$ (with $t_{\rho(n+1)}=t$) and
\begin{align*}
\psi_{t,z,n}^{H}(\pmb{t_n})=c_H^n \int_{\bR^n}\prod_{j=1}^{n}|\cF G_{u_j}(\xi_1+\ldots+\xi_j)|^2 |1-e^{-i(\xi_1+\ldots+\xi_n)z}|^2 \prod_{j=1}^{n}|\xi_j|^{1-2H}d\pmb{\xi_n}.
\end{align*}
Using the inequality $|1-e^{ix}|^2\leq |x|^{2\delta}$ for any $\delta \in [0,1]$ and $x \in \bR$, followed by the change of variables $\eta_j=\xi_1+\ldots+\xi_j$ for $j=1,\ldots,n$ (with $\eta_0=0$) and inequality \eqref{prod-ineq}, we obtain:
\begin{align}
\label{est-psi}
\psi_{t,z,n}^{H}(\pmb{t_n}) &\leq |z|^{2\delta} c_H^n \sum_{\pmb{\alpha_n} \in D_n^{(H)}}\prod_{j=1}^{n-1} \left( \int_{\bR}|\cF G_{u_j}(\eta_j)|^2 |\eta_j|^{\alpha_j}d\eta_j \right)
 \left( \int_{\bR}|\cF G_{u_n}(\eta_n)|^2
 |\eta_n|^{\alpha_n+2\delta}d\eta_n \right).
\end{align}
(We will use \eqref{est-psi} for both heat and wave equations.) Using Lemma \ref{rough-int}, we obtain:
\begin{align*}
\psi_{t,z,n}^{H}(\pmb{t_n}) & \leq |z|^{2\delta} c_H^n \sum_{\pmb{\alpha_n} \in D_n^{(H)}}
 \prod_{j=1}^{n-1} \Gamma\left(\frac{1+\alpha_j}{2}\right)u_j^{-\frac{1+\alpha_j}{2}}
 \Gamma\left(\frac{1+\alpha_n+2\delta}{2}\right)
 u_n^{-\frac{1+\alpha_n+2\delta}{2}}.
\end{align*}
As noticed above, $\prod_{j=1}^{n-1} \Gamma\left(\frac{1+\alpha_j}{2}\right) \leq c^{n-1}$. Moreover, since $\frac{1+\alpha_n+2\delta}{2} \in (\frac{1}{2},H)$, we have $\Gamma\left(\frac{1+\alpha_n+2\delta}{2}\right) \leq \Gamma\left(\frac{1}{2}\right)$. Hence,
\begin{align*}
\psi_{t,z,n}^{H}(\pmb{t_n})\leq |z|^{2\delta} c^n \sum_{\pmb{\alpha_n} \in D_n^{H}} \prod_{j=1}^{n-1}(t_{\rho(j+1)}-t_{\rho(j)})^{-\frac{1+\alpha_j}{2}}
(t-t_{\rho(n)})^{-\frac{1+\alpha_n+2\delta}{2}}.
\end{align*}
Taking power $\frac{1}{2H_0}$ and integrating over $[0,t]^{n}$, we obtain:
\begin{align*}
\int_{[0,t]^n}\psi_{t,z,n}^{H}(\pmb{t_n})^{\frac{1}{2H_0}}d\pmb{t_n} & \leq |z|^{\frac{\delta}{H_0}}n! \sum_{\pmb{\alpha_n} \in D_n^{H}}\int_{T_n(t)} \prod_{j=1}^{n}(t_{j+1}-t_j)^{-\frac{1+\alpha_j}{4H_0}}
(t-t_n)^{-\frac{1+\alpha_n+2\delta}{4H_0}}.
\end{align*}
Note that the sum appearing in the previous estimate is the same as in \eqref{psi-b} in which we replace $\e$ by $\delta/2$ in the exponent of $t-t_n$.
Exactly as \eqref{bound-A-H}, we deduce that
\[
\sum_{n\geq 1}(p-1)^{n/2}\left(\frac{1}{n!}C_{n}^{H}(t,z) \right)^{1/2}\leq C |z|^{\delta}.
\]

\medskip

{\bf Study of $C_n^{H}(t,z)$ (wave equation).} By the Cauchy-Schwarz inequality,
\[
C_{n}^{H}(t,z) \leq \Gamma_{0,t}^n \int_{[0,t]^n}
\psi_{t,z,n}^{H}(\pmb{t_n}) d\pmb{t_n},
\]
with $\Gamma_{0,t}=2H_0 t^{2H_0-1}$ and the same function $\psi_{t,z,n}^{H}(\pmb{t_n})$ as above. To estimate $\psi_{t,z,n}^{H}(\pmb{t_n})$, we use \eqref{est-psi} and Lemma \ref{rough-int}. We obtain:
\begin{align*}
\psi_{t,z,n}^{H}(\pmb{t_n}) \leq |z|^{2\delta}c_H^n \sum_{\pmb{\alpha_n}\in D_n^{(H)}}\prod_{j=1}^{n-1}\left( 2^{1-\alpha_j}\widetilde{C}_{\alpha_j}u_j^{1-\alpha_j} \right)\cdot \left(
2^{1-\alpha_n-2\delta}\widetilde{C}_{\alpha_n+2\delta}u_n^{1-\alpha_n-2\delta}. \right)
\end{align*}
For the application of Lemma \ref{rough-int}, we need $\alpha_n+2\delta<1$, which introduces the restriction $0<\delta<H$ when $\alpha_n=1-2H$. To ensure that this holds for all $H \in [a,b]$, we take $0<\delta<a$. As noticed above,
$\prod_{j=1}^{n-1}2^{1-\alpha_j}\widetilde{C}_{\alpha_j} \leq c^{n-1}$. Moreover, since $\alpha_n+2\delta \in (0,1)$,
\[
\widetilde{C}_{\alpha_n+2\delta}=\frac{1}{1-\alpha_n-2\delta}
\Gamma(\alpha_n+2\delta)\sin \frac{\pi(\alpha_n+2\delta)}{2}.
\]
We study the range of values of $\alpha_n+2\delta$. If $\alpha_n=1-2H$, then $\alpha_n+2\delta \geq 1-2H\geq 1-2b$ and $\Gamma(\alpha_n+2\delta) \leq \Gamma(1-2b)$. The problem is when $\alpha_n=0$ since in this case $\alpha_n+2\delta=2\delta$ may be close to $0$. To avoid this problem, we fix an arbitrary value $c_1>0$, and we choose $\delta \in (\frac{c_1}{2},a)$. Then $\Gamma(\alpha_n+2\delta)\leq \max\{\Gamma(1-2b),\Gamma(c_1)\}$ and $\frac{1}{1-\alpha_n-2\delta}\leq \frac{1}{2a-c_1}$. Hence,
\[
\psi_{t,z,n}^{H}(\pmb{t_n}) \leq |z|^{2\delta}c^n \sum_{\pmb{\alpha_n}\in D_n^{(H)}}\prod_{j=1}^{n-1}(t_{\rho(j+1)}-t_{\rho(j)})^{1-\alpha_j}
(t-t_{\rho(n)})^{1-\alpha_n-2\delta}.
\]
Integrating over $[0,t]^n$, we obtain:
\begin{align*}
\int_{[0,t]^n}\psi_{t,z,n}^{H}(\pmb{t_n}) d\pmb{t_n} & \leq
|z|^{2\delta}c^{n-1} n! \sum_{\pmb{\alpha_n}\in D_n^{(H)}}\int_{T_n(t)}\prod_{j=1}^{n-1}(t_{j+1}-t_j)^{1-\alpha_j}
(t-t_n)^{1-\alpha_n-2\delta}d\pmb{t_n}\\
&=|z|^{2\delta}c^{n-1} n! \sum_{\pmb{\alpha_n}\in D_n^{(H)}}\int_0^t J^w(t_n)(t-t_n)^{1-\alpha_n-2\delta} dt_n,
\end{align*}
where $J^w$ is defined and estimated in \eqref{def-Jw}. Using this estimate for $J^w$, we see that:
\[
\int_{[0,t]^n}\psi_{t,z,n}^{H}(\pmb{t_n}) d\pmb{t_n} \leq  |z|^{2\delta}c^{n-1} \frac{1}{[(n-1)!]^{2a}}.
\]
We conclude that
\[
\sum_{n\geq 1}(p-1)^{n/2}\left( \frac{1}{n!}C_n^H(t,z)\right)^{1/2}\leq c |z|^{\delta}.
\]

\end{proof}

\bigskip

{\bf Proof of Theorem \ref{main-th2}:}

{\em Step 1. (finite dimensional convergence)} In this step, we prove that:
\[
\big(u^{H_n}(t_1,x_1),\ldots,u^{H_n}(t_k,x_k)\big)
\stackrel{d}{\to}
\big(u^{H^*}(t_1,x_1),\ldots,u^{H^*}(t_k,x_k)\big),
\]
for any $(t_1,x_1),\ldots,
(t_k,x_k)\in [0,T] \times \bR$. It is enough to prove that for any $(t,x) \in [0,T] \times \bR$, $u^{H_n}(t,x) \to u^{H^*}(t,x)$ in $L^2(\Omega)$ as $n\to \infty$.

As in the proof of Theorem \ref{main-th1}, we approximate $u^{H}(t,x)$ by the partial sum:
\[
u_m^H(t,x)=1+\sum_{k=1}^m I_{n}^H(f_{t,x,k}).
\]
Then $u_m^H(t,x) \to u^H(t,x) \to 0$ in $L^2(\Omega)$ as $m\to \infty$. Moreover, by Lemma \ref{rough-conv-Ik}, for any $m\geq 1$ fixed,
$u_m^{H_n}(t,x) \to u_m^{H*}(t,x)$ in $L^2(\Omega)$ as $n\to \infty$. So, it remains to prove that $\sup_{n\geq 1}\bE |u_m^{H_n}(t,x)-u^H(t,x)|^2 \to 0$ as $m\to \infty$. We choose values $a$ and $b$ such that \eqref{def-ab} holds. Since $H_n \to H^*$, there exists $N \in \bN$ such that $a<H_n<b$ for all $n\geq N$. Therefore, it suffices to prove that:
\begin{equation}
\label{rough-unif-conv}
\sup_{H \in [a,b]}\bE |u_m^H(t,x)-u^H(t,x)|^2=\sup_{H \in [a,b]}\sum_{k\geq m+1}\bE|I_k^{H}(f_{t,x,k})|^2 \to 0 \quad \mbox{as $m\to \infty$}.
\end{equation}

To prove \eqref{rough-unif-conv}, we need to estimate $\bE|I_k^{H}(f_{t,x,k})|^2$. Although these estimates exist in the literature, we revisit them here since we are interested in a uniform bound in $H$.
First, note that by Littlewood-Hardy inequality \eqref{LH-ineq},
\begin{equation}
\label{bound-Ik}
\bE|I_k^{H}(f_{t,x,k})|^2 =k! \, \|\widetilde{f}_{t,x,k} \|_{\cH_{H}^{\otimes k}}^2 \leq k! \, b_{0}^k \left(\int_{[0,t]^k}
A_k^{H}(\pmb{t_k})^{\frac{1}{2H_0}}d\pmb{t_k}\right)^{2H_0},
\end{equation}
where $A_k^{H}(\pmb{t_k})=c_H^k \int_{\bR^k}|\cF \widetilde{f}_{t,x,k}(\pmb{t_k},\bullet)(\pmb{\xi_k})|^2 \prod_{j=1}^{k}|\xi_j|^{1-2H}d\pmb{\xi_k}$.

Next, we would like to find an upper bound for $A_k^{H}(\pmb{t_k})$.
The major difference compared with the regular case is that in the rough case, we cannot apply Lemma \ref{G-lemma}. Instead, we will use the common technique which consists of the change of variables $\eta_j=\xi_1+\ldots+\xi_j$ for $j=1,\ldots,k$ (with $\eta_0=0$), followed by inequality \eqref{prod-ineq} for bounding a product with a sum. In the application of \eqref{prod-ineq} below, we let $D_k^{(H)}$ be the set of multi-indices $\pmb{\alpha}=(\alpha_1,\ldots,\alpha_k)$ with $\alpha_j=(1-2H)a_j$ and $\pmb{a}=(a_1,\ldots,a_k) \in A_k$.
Hence, if $\pmb{t_k}=(t_1,\ldots,t_k) \in [0,t]^k$ and  $\rho$ is a permutation of $1,\ldots,k$ such that $t_{\rho(1)}<\ldots<t_{\rho(k)}$, with $t_{\rho(k+1)}=t$, then
\begin{align*}
A_k^{H}(\pmb{t_k}) &=\frac{c_H^k}{(k!)^2}\int_{\bR^k} \prod_{j=1}^{k} |\cF G_{t_{\rho(j+1)}-t_{\rho(j)}}(\xi_1+\ldots+\xi_j)|^2
\prod_{j=1}^{k}|\xi_j|^{1-2H}d\pmb{\xi_k}\\
&=\frac{c_H^k}{(k!)^2}\int_{\bR^k} \prod_{j=1}^{k}
|\cF G_{t_{\rho(j+1)}-t_{\rho(j)}}(\eta_j)|^2
\prod_{j=1}^{k}|\eta_j-\eta_{j-1}|^{1-2H}d\pmb{\xi_k}\\
&\leq \frac{c_H^k}{(k!)^2} \sum_{\pmb{\alpha}\in D_k^{(H)}} \prod_{j=1}^{k}\left(\int_{\bR}
|\cF G_{t_{\rho(j+1)}-t_{\rho(j)}}(\eta_j)|^2 |\eta_j|^{\alpha_j} d\eta_j\right).
\end{align*}

The integrals above are evaluated using Lemma \ref{rough-int}.  We obtain:
\begin{align*}
A_k^{H}(\pmb{t_k}) &\leq \frac{c_H^k}{(k!)^2} \sum_{\pmb{\alpha}\in D_k^{(H)}} \prod_{j=1}^{k}
h(\alpha_j)(t_{\rho(j+1)}-t_{\rho(j)})^{\varphi(\alpha_j)},
\end{align*}
where
\[
h(\alpha_j)=
\left\{
\begin{array}{ll}
\Gamma((1+\alpha_j)/2)  & \mbox{for heat equation} \\
2^{1-\alpha_j} \widetilde{C}_{\alpha_j} & \mbox{for wave equation}
\end{array} \right.
\quad
\varphi(\alpha_j)=
\left\{
\begin{array}{ll}
-(1+\alpha_j)/2  & \mbox{for heat equation} \\
1-\alpha_j & \mbox{for wave equation}
\end{array} \right.
\]
Since $\alpha_j \in \{0,1-2H,2(1-2H)\}$, $h(\alpha_j)\leq C_{H,1}$, where the constant $C_{H,1}$ is given by
\eqref{def-CH1} for the heat equation, respectively \eqref{def-CH1-w} for the wave equation.
Coming back to \eqref{bound-Ik}, we obtain that:
\begin{align*}
\bE|I_k^H(f_{t,x,k})|^2 & \leq k!\frac{c_H^k}{(k!)^2}b_{H_0}^k C_{H,1}^k
\left[k! \int_{T_k(t)} \left(\sum_{\pmb{\alpha} \in D_k}
\prod_{j=1}^{k}(t_{j+1}-t_{j})^{\varphi(\alpha_j)} \right)^{\frac{1}{2H_0}}d\pmb{t_k}
\right]^{2H_0} \\
& \leq
(k!)^{2H_0-1} c_H^k b_{H_0}^k C_{H,1}^k
\left[ \sum_{\pmb{\alpha} \in D_k} \int_{T_k(t)} \prod_{j=1}^{k}(t_{j+1}-t_{j})^{\frac{\varphi(\alpha_j)}{2H_0}} d\pmb{t_k} \right]^{2H_0}.
\end{align*}
The last integral is calculated using Lemma \ref{beta-lem}: in the case of the heat equation, this integral is:
\[
\int_{T_k(t)} \prod_{j=1}^{k}(t_{j+1}-t_{j})^{-\frac{1+\alpha_j}{4H_0}}d\pmb{t_k} =\frac{\prod_{j=1}^{k}
\Gamma(-\frac{1+\alpha_j}{4H_0}+1)}{\Gamma(k\frac{2H_0+H-1}{2H_0}+1)}
t^{k\frac{2H_0+H-1}{2H_0}}\leq C_{H,2}^k \frac{t^{k\frac{2H_0+H-1}{2H_0}}}{\Gamma(k\frac{2H_0+H-1}{2H_0}+1)},
\]
and in the case of the wave equation, the integral is:
\[
\int_{T_k(t)} \prod_{j=1}^{k}(t_{j+1}-t_{j})^{\frac{1-\alpha_j}{2H_0}}d\pmb{t_k} =\frac{\prod_{j=1}^{k}
\Gamma(\frac{1-\alpha_j}{2H_0}+1)}{\Gamma(k\frac{H_0+H}{H_0}+1)}
t^{k\frac{H_0+H}{H_0}}\leq C_{H,2}^k \frac{t^{k\frac{H_0+H}{H_0}}}{\Gamma(k\frac{H_0+H}{H_0}+1)},
\]
where the constant $C_{H,2}$ is given by:
\[
C_{H,2}=
\left\{
\begin{array}{ll}
\max\{\Gamma(1-\frac{1}{4H_0}),\Gamma(1-\frac{1-H}{2H_0}),
\Gamma(1-\frac{3-4H}{4H_0}))\} & \mbox{for heat equation} \\
\max\{\Gamma(1+\frac{1}{2H_0}),\Gamma(1+\frac{H}{H_0}),
\Gamma(1+\frac{4H-1}{2H_0})) \} & \mbox{for wave equation}
\end{array} \right.
\]
It is not difficult to see that these expressions can be uniformly bounded for all $H \in [a,b]$. This concludes the proof of \eqref{rough-unif-conv}.

\medskip

{\em Step 2. (tightness)} The fact that $(u^{H_n})_{n\geq 1}$ is tight in $C([0,T] \times \bR)$ follows by Proposition 2.3 of \cite{yor} using Theorem \ref{rough-unif-mom} above.

\end{document}